\documentclass[12pt]{article}

\usepackage{amsfonts}
\usepackage{amsmath}
\usepackage{amssymb}
\usepackage{amscd}
\usepackage{amsthm}
\usepackage{indentfirst}
\usepackage{MnSymbol}
\usepackage{stackrel}
\usepackage[vmargin=3.4cm,hmargin=2.3cm]{geometry}
\usepackage{enumerate}

\newtheorem{theorem}{Theorem}[section]
\newtheorem{corollary}[theorem]{Corollary}
\newtheorem{lemma}[theorem]{Lemma}
\newtheorem{proposition}[theorem]{Proposition}

\newtheorem*{theorem no number}{Theorem}

\theoremstyle{definition}

\newtheorem{definition}[theorem]{Definition}

\newtheorem*{acknow}{Acknowledgements}
\newtheorem*{notation}{Notation}
\newtheorem*{convention}{Convention}
\newtheorem{remark}[theorem]{Remark}

\newtheorem*{remark no number}{Remark}

\theoremstyle{remark}

\newtheorem*{claim no number}{Claim}

\newcommand{\R}{{\mathbb{R}}}

\newcommand{\Z}{{\mathbb{Z}}}
\newcommand{\N}{{\mathbb{N}}}

\usepackage{zref-xr}

\usepackage[backend=bibtex,style=alphabetic,bibencoding=utf8,language=auto,autolang=other,giveninits=true,doi=false,isbn=false,url=false,maxnames=10,sorting=nty]{biblatex}

\usepackage{tikz-cd}
\usetikzlibrary{matrix,arrows,decorations.pathmorphing}

\usepackage{comment}

\newcommand\restr[2]{{
  \left.\kern-\nulldelimiterspace 
  #1 
  \vphantom{\big|} 
  \right|_{#2} 
  }}
  
\usepackage{hyperref}
\usepackage{url}
\usepackage{graphicx}

\title{\textbf{A Bangert--Hingston Theorem for Starshaped Hypersurfaces}}
\author{Alessio Pellegrini \\ \\ \textit{Department of Mathematics, ETH Z\"urich, Switzerland} \\ alessio.pellegrini@math.ethz.ch}

\begin{document}

\maketitle \date

\abstract{Let $Q$ be a closed manifold with non-trivial first Betti number that admits a non-trivial $S^1$-action, and $\Sigma \subseteq T^*Q$ a non-degenerate starshaped hypersurface. We prove that the number of geometrically distinct Reeb orbits of period at most $T$ on $\Sigma$ grows at least logarithmically in $T$.}

\section{Introduction}

In a celebrated paper of Bangert and Hingston \cite{bh1984} it is shown that the number of geometrically distinct closed geodesics on a closed connected manifold $Q$ grows like the prime numbers whenever $\pi_1(Q) \cong \Z$ and $\dim Q \geq 2$. More precisely, if $\mathcal{N}(T)$ denotes the number of geometrically distinct geodesics of period at most $T$, then $$\liminf_{T \to \infty}\mathcal{N}(T) \cdot \frac{\log(T)}{T}>0.$$ Bangert and Hingston's proof makes use of minimax values associated to the energy functional $$\mathcal{E} \, \colon \mathcal{L}Q \to \R, \; \mathcal{E}(q)=\int_0^1\frac{1}{2}  \Vert \dot{q}(t) \Vert^2 \, dt$$ and two sequences of homotopy classes on the free loop space $\mathcal{L}Q$. Thanks to the Legendre transform, critical points of $\mathcal{E}$, i.e.\ geodesics, are in one to one correspondence with the critical points of the Hamiltonian action functional $\mathcal{A}_H \, \colon \mathcal{L}T^*Q \to \R$ associated to the kinetic Hamiltonian $H(q,p)=\frac{1}{2} \Vert  p \Vert^2$. In particular, the quest for closed geodesics can be rephrased as a Hamiltonian problem, and if one wishes to fix the constant speed of the geodesics to be $1$, the problem becomes a quest for periodic Reeb orbits on the unit sphere bundle $$H^{-1}\big(1/2\big)=S^*Q.$$ 

The Hamiltonian, Lagrangian, and Reeb approach are essentially equivalent and are still interchangeable for general convex hypersurfaces in $T^*Q$.  However, if $\Sigma \subseteq T^*Q$ is \textbf{starshaped}, there is no Legendre transform in general and therefore the notion of Reeb orbits does not admit a Lagrangian reformulation. Nevertheless, the attempt to carry over well-known geodesic growth type results to the realm of contact geometry on $\Sigma$ has been quite successful. Indeed, many striking results where the topology of $Q$ (or its loop space $\mathcal{L}Q$) forces certain geodesic growths have found an analogous formulation for growth rates of Reeb orbits/chords/leaf-wise intersections on $\Sigma$ \cite{ms2011, heistercamp2011, mmp2012, wullschleger2014}. A common theme throughout these results is the use of some flavour of Floer theory and its relation to the singular homology of the loop space $\mathcal{L}Q$.

Prior to the work of Bangert and Hingston, several geodesic growth results had been established in the finite fundamental group case
 \cite{gm1969, gromov1978, btz1981, bz1982}, while in the infinite abelian case very little was known, especially for the ``smallest" case, i.e.\ $\pi_1(Q) \cong \Z$. The starshaped setting encounters the same issue, as for $\pi_1(Q)=\Z$ there is no known result, to the best of the author's knowledge, on the growth of periodic Reeb orbits of $\Sigma \subseteq T^*Q$. More precisely, if $\mathcal{N}_\Sigma(T)$ denotes the number of geometrically distinct periodic Reeb orbits of period at most $T$, then the behaviour of $\mathcal{N}_\Sigma(T)$, as $T$ goes to infinity, is currently unknown. The following result partially fills this gap.

\begin{theorem no number}[Main Theorem]\label{theorem Main Theorem Introduction}
 Let $Q$ be a $n$-dimensional closed connected manifold with $n \geq 2$, whose first Betti number is non-trivial. Assume that $Q$ admits a non-trivial $S^1$-action.
 Then for any non-degenerate starshaped hypersurface $\Sigma \subseteq T^*Q$ it holds:
 $$\liminf_{T \to \infty}\frac{\mathcal{N}_\Sigma(T)}{\log(T)} >0.$$
 In particular, $\Sigma$ admits infinitely many geometrically distinct Reeb orbits.
 \end{theorem no number}

The conclusion of the Main Theorem holds for a number of examples, e.g. products $Q=S^1 \times M$ with $M$ a closed manifold, $H$-spaces $Q$ with infinite fundamental group, and principal $S^1$-bundles $Q$ over a closed connected base manifold $B$ with $\pi_2(B)=0$ --- see Theorem \ref{theorem main section existence} and Corollary \ref{corollary Sigma non-degenerate} for more details.

We will now present an outline of the proof of the Main Theorem and explain the various assumptions as we go along. The main ingredient is the use of \textbf{spectral invariants} $$c_\alpha(H) \in \R, \quad \alpha \in H_\bullet(\mathcal{L}Q)\setminus\{0\},$$
where $H \, \colon T^*Q \to \R$ is an autonomous asymptotic quadratic Hamiltonian in the sense of Abbondandolo--Schwarz \cite{as2006}. Let $F \, \colon T^*Q \to \R$ be a quadratic Hamiltonian with $\Sigma=F^{-1}(1)$, and let $G$ be a kinetic Hamiltonian together with a constant $\sigma>1$ such that
$$G \leq F \leq \sigma G.$$
Then the spectral invariants are related to each other via
$$\frac{1}{\sigma} \cdot c_\alpha(G) \leq c_\alpha(F) \leq c_\alpha(G).$$ This is based on a \textbf{pinching argument} due to Macarini and Schlenk \cite{ms2011}, which also plays a key role in the arguments of the aforementioned papers \cite{heistercamp2011, mmp2012, wullschleger2014}. 

The Betti number assumption ensures the existence of a homotopy class $\eta \in \pi_1(Q)$, whose image in $H_1(Q;\Z)$ is of infinite order. Denote by $\mathcal{L}_mQ$ the connected component of the loop space associated to $\eta^m$ for $m \in \Z$. We say that a smooth $S^1$-action $\phi \, \colon S^1 \times Q \to Q$ is \emph{non-trivial} with respect to $\eta$, if the loop $\gamma_q(t):=\phi(t,q)$ belongs to $\mathcal{L}_kQ$ for some integer $k \neq 0$.\footnote{Note that any two orbits of $\phi$ have the same free homotopy class since $Q$ is assumed to be connected.} In particular, the map
$$s \, \colon Q \longrightarrow \mathcal{L}_kQ, \; s(q):= \gamma_q$$
defines a section of the evaluation map $\mathrm{ev}_0 \, \colon \mathcal{L}Q \to Q$. The existence of the section enables us to define a sequence of non-zero homology classes $\alpha_m$ on pairwise distinct components of $\mathcal{L}Q$ as follows: define the $\mathbf{m}$\textbf{-iteration map}
$$\mathcal{I}_m \, \colon \mathcal{L} Q \longrightarrow \mathcal{L}Q, \quad \gamma \mapsto \underbrace{\gamma * \dots * \gamma}_{m\text{-times}}.$$
Observe that the map 
$$\mathcal{I}_m \circ s  \, \colon Q \longrightarrow \mathcal{L} Q$$
maps into $\mathcal{L}_{mk}Q$ and is still a section of $\mathrm{ev}_0 \, \colon \mathcal{L}Q \to Q$.
By functoriality we therefore have a commutative diagram of the form

\begin{center}
\begin{tikzcd}
H_n(Q) \arrow[r,"(s)_*"] \arrow[rrr,"\mathrm{id}",bend right=25] & H_n(\mathcal{L}_k Q) \arrow[r,"(\mathcal{I}_m)_*"] & H_n(\mathcal{L}_{mk}Q) \arrow[r,"(\mathrm{ev}_0)_*"] & H_n(Q) 
\end{tikzcd}
\end{center}
Since $H_n(Q) \neq0$,\footnote{We are implicitly working with $\Z_2$-coefficients.} there exists a non-zero cohomology class $\alpha_1 \in H_n(\mathcal{L}_kQ)$, together with a sequence of non-zero homology classes
$$ \alpha_m\, :=(\mathcal{I}_m)_*\alpha_1 \in H_n(\mathcal{L}_{mk}Q) \setminus \{0\}, \quad \forall m \in \N.$$
By spectrality we know that there exists a sequence of Hamiltonian orbits $x_m$ of $F$ satisfying $$c_{\alpha_m}(F)=\mathcal{A}_F(x_m).$$ These orbits $x_m$ correspond to periodic Reeb orbits on $\Sigma$ by choice of $F$.
Using a Robbin--Salamon \textbf{index growth result} due to M. de Gosson, S. de Gosson and Piccione \cite{ggp2008} we can extract a subsequence, still denoted by $x_m$, such that the $x_m$ are not iterates of each other --- this is the only part that uses the non-degeneracy of $\Sigma$. The precise growth rate of the Reeb orbits $x_m$ in terms of their action (and thus their period) is determined by using the spectral invariant inequality from above. This concludes the outline of the proof.

Under the $S^1$-action assumption, Irie's Main Theorem \cite{irie2014} on the finiteness of the Hofer--Zehnder capacity of disk cotangent bundles $D^*Q$ is applicable, however this is not particularly fruitful when applied to the function $F$ due to homogeneity: the finiteness of the Hofer--Zehnder capacity of $D^*Q$ implies dense existence of orbits nearby $\Sigma=F^{-1}(1)$ \cite{hz2012}, but it could happen that these orbits all correspond to a single Reeb orbit on $\Sigma$. To make matters worse, no growth rate control can be deduced even if we knew that there are infinitely many Reeb orbits. In this sense, the Main Theorem can be viewed as a strengthening to the dynamical consequences of Irie's Theorem in the case of starshaped hypersurfaces.

Let us mention that there is no need to bound the pinching factor $\sigma$ in our arguments. This seems to be a reoccurring  dichotomy between starshaped hypersurfaces in $T^*Q$ (with $Q$ closed) and (compact) starshaped hypersurfaces in $\R^{2n}$, where for the latter the pinching factor $\sigma$ plays a more crucial role \cite{girardi1984,blmr1985,viterbo1989,ekeland2012,agd2016,wang2016, dl2017,am2017}.

The methods in BH (Bangert and Hingston) \cite{bh1984} are not well suited to our setting since spectral invariants can only be assigned to \emph{homology classes}, while in the Lagrangian case, as done in BH, minimax values can be assigned to homotopy classes as well. This issue can be ascribed to the lack of a global flow for the negative gradient of the Hamiltonian action functional, which is in contrast to the existence of a global flow of the negative gradient of the Lagrangian action functional. Even under further assumptions, e.g. that the relevant homotopy groups inject into the free loop space homology via the Hurewicz homomorphism, we have not been able to adapt BH's arguments.

\begin{remark}
Let us emphasise that non-degeneracy of $\Sigma$ is a $C^\infty$-generic property; cf.\ \cite[Theorem 2.5]{mmp2012}. While in the $C^1$-generic case stronger dynamical conclusions than Theorem \ref{theorem Main Theorem Introduction} can be drawn without any assumptions on $\Sigma$, for instance using results due to Newhouse \cite[Section 5 and 6]{newhouse2020} and Smale \cite{smale2015}, yielding a horseshoe and hence exponential growth of the closed orbits,  the existence of infinitely many orbits seems to be unknown in the general $C^\infty$-generic case.
\end{remark}

\begin{remark}
 In the case that $\Sigma$ is convex, the Lagrangian formulation is sufficient to run the proof strategy of \cite{bh1984}, also see \cite{bj2008} for the Finsler case. This gives a better growth rate, namely $T / \log(T)$, without any non-degeneracy assumptions. The present work is part of the author's PhD thesis, in which we also address the case of $\Sigma$ convex without the non-degeneracy assumption --- this gives some further insights into the key differences between starshapedness and convexity.
\end{remark}

\begin{acknow}
I would like to thank Will Merry for all the helpful discussions and for suggesting this project, Marco Mazzucchelli and Felix Schlenk for many helpful comments, and the anonymous referee for helping improve the whole exposition. This work has been partially supported by the Swiss National Science Foundation (grant $\# 182564$).
\end{acknow}

\section{Reeb Orbits and Starshaped Domains}\label{section Reeb Orbits on Starshaped Domains}

Let $Q$ be a closed connected $n$-dimensional manifold. The $2n$-dimensional cotangent bundle $T^*Q$ is equipped with the symplectic form $\omega=d\lambda$
where $\lambda$ is the Liouville form that is locally given by $p \, dq$. For any Hamiltonian $$H_t \, \colon T^*Q \longrightarrow \R$$ we define the Hamiltonian vector field $X_{H_t}$ via
$$\omega(\cdot,X_{H_t})=dH_t.$$  Moreover, an almost complex structure $J$ on the manifold $T^*Q$ is said to be compatible if $g_J\, :=\omega( \, \cdot \, , J \cdot \,)$ defines a Riemannian metric. We define the Hamiltonian action functional $\mathcal{A}_H$ on loops $x\, \colon S^1 \to T^*Q$:
$$\mathcal{A}_{H}(x)\, :=\int_x \lambda - \int_0^1 H_t(x(t)) \, dt.\footnote{The whole sign convention agrees with the one of \cite{as2006}.}  $$
The set of $1$-periodic Hamiltonian orbits of $H_t$ is denoted by $\mathcal{P}(H)$, and $\mathcal{P}^{I}(H)$ denotes the subset of $\mathcal{P}(H)$ of orbits $x$ with action $\mathcal{A}_H(x) \in I \subseteq \R$. For $(-\infty,d]$ we abbreviate
$\mathcal{P}^d(H)=\mathcal{P}^{(-\infty,d]}(H)$.

Let $\Sigma \subseteq T^*Q$ be a smooth, connected hypersurface with $\Sigma=\partial D$, where $D \subseteq T^*Q$ is a bounded domain $D$ that contains the zero section. Such a $\Sigma$ is called \textbf{starshaped}, if in each fibre $T_q^*Q$ the set $\Sigma_q\, :=T_q^*Q \cap \Sigma$ is starshaped with respect to the origin $0_q \in T^*_qQ$. Observe that $\Sigma$ is necessarily compact. 

Any starshaped $\Sigma$ is of restricted contact type with a contact form $\alpha_\Sigma$ defined as follows: for each point $q$ and $p \in T_q^*Q$ one can consider the path $\gamma_{(q,p)}(t)=t \cdot p \in T_q^*Q $ and define the Liouville vector field via
$$Y(q,p)\, :=\dot{\gamma}_{(q,p)}(1) \in T_{(q,p)}T^*Q.\footnote{Indeed, one can check that locally $Y=p\frac{\partial}{\partial p}$, thus $\omega(Y,\cdot)=p \, dq=\lambda$.}$$  By definition, $Y|_\Sigma$ is outward pointing and transverse to $\Sigma$, thus implying that $\Sigma$ is of restricted contact type with contact form 
$$\alpha_\Sigma\, :=\omega(Y, \cdot \, |_{T\Sigma})=\lambda |_{T\Sigma}$$ 
and contact structure 
$$\xi\, :=\ker \alpha_\Sigma.$$

Let us translate the quest of finding Reeb orbits on $(\Sigma,\xi)$ into a Hamiltonian problem. Since $\Sigma$ is assumed to be starshaped, there exists a $2$-homogeneous Hamiltonian $$F \, \colon T^*Q \longrightarrow \R$$ that is uniquely defined by requiring
$$F^{-1}(1)=\Sigma, \quad F(q,sp)=s^2F(q,p), \quad \forall s \geq 0, \, \forall (q,p) \in T^*Q.\footnote{While $F$ is smooth away from the zero section, we can only expect $C^1$ regularity near $0$ --- we will remedy this issue later by composing $F$ with an auxiliary function $f \, \colon \R \to \R$ that smooths out $F$ without affecting the set of non-constant $1$-periodic orbits; cf.\ Proposition \ref{proposition precomposing with f does not change orbits}.}$$ Furthermore,
$$\iota_{X_F |_\Sigma}d\alpha_\Sigma=\iota_{X_F |_\Sigma} d\left( \lambda |_{T\Sigma}\right)=-dF|_{T\Sigma}$$ and since $F$ is constant along $\Sigma$ we have $\iota_{X_F |_\Sigma}d\alpha_\Sigma=0$. Similarly, we know that $X_F(x)$ is a vector in $T_x\Sigma$ for all $x \in \Sigma$, in particular using homogeneity of $F$ we deduce
$$\alpha_\Sigma(X_F(x))=\lambda(X_F(x))=2 \cdot F(x) \equiv 2, \quad \forall x \in \Sigma,$$
and with the above we therefore obtain  $$X_F |_\Sigma= 2 \cdot R.$$
In particular, for $\varphi_F^t$ (resp. $\varphi_R^t)$ the time-$t$ Hamiltonian flow of $F$ (resp. Reeb flow) we get
$$\varphi^t_F(x)=\varphi_R^{2t}(x), \quad \forall x \in \Sigma, \, t \in \R.$$

\begin{definition}
Denote $\mathcal{O}_R(t)$ the set of Reeb orbits of period less or equal $t$.\footnote{We identify two orbits that are equal up to a time-shift.} Similarly, $\mathcal{O}_{X_F |_\Sigma}(t)$ (resp. $\mathcal{O}_{\mathcal{A}_F}(t)$) denotes the set of Hamiltonian orbits of $X_F |_\Sigma$ (resp. $X_F$) of period less or equal $t$ (resp. action less or equal $t$).
\end{definition}
The above flow relation implies
 $$\# \mathcal{O}_{X_F|_\Sigma}(t) = \# \mathcal{O}_R(2t).$$

When doing Floer homology, we will consider only $1$-periodic orbits, but across the whole cotangent bundle. The next proposition shows that fixing the level and letting the period vary has the same effect as fixing the period and varying the level. First we need some notation. Let $x=(q,p) \in T^*Q$. Then for every $s \in \R$ define
$$s  \cdot x=(q,sp) \in T^*Q.$$
For a loop $x=(q,p)$ on $T^*Q$ define
$$x_s(t)\, :=s\cdot x(st)=(q(st),s \cdot p(st)).$$

The following result is contained in \cite{heistercamp2011}.

\begin{proposition}\label{proposition 1-per somewhere is same as per but on sigma}
There is a one to one correspondence between $1$-periodic orbits $x \, \colon S^1 \to T^*Q$ of $X_F$ with action $\mathcal{A}_F(x)=a$, and $\sqrt{a}$-periodic orbits $y$ of $X_F |_\Sigma$. The correspondence is given by
$$x \mapsto x_{1 / \sqrt{a}}, \quad y \mapsto y_{\sqrt{a}}.$$
\end{proposition}

Proposition \ref{proposition 1-per somewhere is same as per but on sigma} implies $$\#\mathcal{O}_{\mathcal{A}_F}(a)=\#\mathcal{O}_{X_F |_\Sigma}(\sqrt{a})$$ for all $a>0$. All in all we thus obtain

\begin{equation*}\label{equation number of Reeb and Ham orbits}
\#\mathcal{O}_R(2t) =  \# \mathcal{O}_{X_F |_\Sigma}(t)= \# \mathcal{O}_{\mathcal{A}_F}\left(t^2\right), \quad \forall t > 0.
\end{equation*}
In particular, we can bound the number of geometrically distinct Reeb $T$-periodic orbits
$$\mathcal{N}_\Sigma(T)\, :=\#\mathcal{O}_R(T)$$
from below by bounding $\# \mathcal{O}_{\mathcal{A}_F}\left((T/2)^2\right)$ from below instead.

\section{Spectral Invariants and Minimax Values}\label{section Spectral Invariants and Minimax Values}

\subsection{Spectral Invariants of Quadratic Hamiltonians}
In this section we will introduce spectral invariants and compare them to minimax values on $T^*Q$. Many of the results about spectral invariants here are very analogous to those in \cite{irie2014}. The main difference is the choice of Hamiltonians --- Irie works with Hamiltonians that are linear at infinity, while we opt for quadratic Hamiltonians. The main advantage of quadratic Hamiltonians is that their Floer homology computes the singular homology of the loop space thanks to the work of Abbondandolo, Majer and Schwarz \cite{as2006, am2006, as2015}. The main reason this works so well is that the relevant chain isomorphisms are action preserving.

\begin{convention}
All homologies are implicitly understood to be over the field $\Z_2$ in order to avoid the need of local coefficients whenever the second Stiefel-Whitney class does not vanish over tori \cite{as2014corrigendum, abouzaid2015}.
\end{convention}

Denote by $H \, \colon S^1 \times T^*Q \to \R$ a  Hamiltonian satisfying 
\begin{itemize}
\item[(H0)] every $x \in \mathcal{P}(H)$ is non-degenerate,
\item[(H1)] $dH(t,q,p)[Y]-H_t(q,p) \geq h_0 \Vert p \Vert^2-h_1$, for some constants $h_0>0$ and $h_1 \geq 0$,
\item[(H2)] $\Vert \nabla_q H(t,q,p) \Vert \leq h_2(1+\Vert p \Vert^2)$ and $\Vert \nabla_p H(t,q,p) \Vert \leq h_2(1+\Vert p \Vert)$, for some constant $h_2 \geq 0$.
\end{itemize}
These are the conditions used by AS (Abbondandolo and Schwarz) to define Floer homology on cotangent bundles. Additionally, assume that $H$ admits a Legendre dual Lagrangian $$L\, \colon S^1 \times TQ \longrightarrow \R$$ (e.g. whenever $H$ is strictly convex). Let $\mathcal{L}Q$ be the free loop space of $Q$. Denote by $$\mathcal{E}_L \, \colon \mathcal{L}Q \longrightarrow \R, \quad \mathcal{E}_L(q)=\int_0^1 L_t(q(t),\dot{q}(t)) \, dt$$ the corresponding Lagrangian action functional and define $$\mathrm{CM}_\bullet(\mathcal{E}_L)$$ as the Morse chain complex of $\mathcal{E}_L$. Denote by $$\mathrm{CM}^d_\bullet(\mathcal{E}_L)=\mathrm{CM}_\bullet^{(-\infty,d]}(\mathcal{E}_L)$$ the filtered Morse chain complex of $\mathcal{E}_L$ associated to a regular value $d \in \R$ (similarly $\mathrm{CF}_\bullet^d(H)$ using the action functional $\mathcal{A}_H$),\footnote{The Floer chain complex $\mathrm{CF}_\bullet(H)$ is graded using the Conley--Zehnder index and ``vertically preserving" trivializations as used in \cite{as2006}.} and write 
$$i^d \, \colon \mathrm{HM}^d_\bullet(\mathcal{E}_L) \longrightarrow \mathrm{HM}_\bullet(\mathcal{E}_L)$$
for the induced map.\footnote{In general the induced map in homology is \emph{not} an inclusion.}

\begin{definition}
Let $\alpha \in \mathrm{HM}_\bullet(L) \setminus \{0\}$ (resp. $\beta \in \mathrm{HF}_\bullet(H)\setminus \{0\})$. The \textbf{minimax value} associated to $\alpha$ and $\mathcal{E}_L$ (resp. the \textbf{spectral invariant} associated to $\beta$ and $H$) is defined as
$$c_\alpha(L)=\inf_{\xi=\sum_{i}\xi^i q_i \in \alpha} \mathcal{E}_L(\xi), \text{ where } \mathcal{E}_L(\xi)=\max_{\xi^i \neq 0}\mathcal{E}_L(q_i),$$
resp.
$$c_\beta(H)=\inf_{\zeta=\sum_{i}\zeta^i x_i \in \beta}\mathcal{A}_H(\zeta), \text{ where } \mathcal{A}_H(\zeta)=\max_{\zeta^i \neq 0} \mathcal{A}_H(x_i).$$
\end{definition}
An equivalent definition is given by
$$c_\alpha(L)=\inf\{ c  \; | \, \alpha \in \mathrm{im}(i^c) \},$$
where the infimum runs over $\R$ minus the set of critical values of $\mathcal{E}_L$. The definitions above beg the question as to whether their is a relation between minimax values and spectral invariants. The answer is yes, but this needs a further digression to the work of \cite{as2006}. Therein, a chain map isomorphism

$$\Theta \, \colon \mathrm{CM}_\bullet(\mathcal{E}_L) \overset{\cong}{\longrightarrow} \mathrm{CF}_\bullet(H),$$
is constructed and shown to respect the action filtration, i.e.\ for every $d$ the map $\Theta$ induces a chain map isomorphism, still denoted by $\Theta$:
$$\Theta \, \colon \mathrm{CM}_\bullet^d(\mathcal{E}_L) \overset{\cong}{\longrightarrow} \mathrm{CF}^d_\bullet(H).$$
In \cite{as2015}, AS came up with another construction yielding a chain isomorphism
$$\Psi \, \colon \mathrm{CF}_\bullet(H) \overset{\cong}{\longrightarrow} \mathrm{CM}_\bullet(\mathcal{E}_L),$$
which is action preserving as well and defines a chain homotopy inverse to $\Theta$, i.e.\ $\Psi \circ \Theta$ and $\Theta \circ \Psi$ are chain homotopy equivalent to the corresponding identity on the chain level. With these at hand we have the following easy, but absolutely crucial proposition.

\begin{proposition}\label{proposition floer and morse minimax agree}
Let $H$ and $L$ be dual and assume that they both satisfy the AS conditions (H0), (H1), (H2). Then for all $\alpha \in \mathrm{HM}_\bullet(\mathcal{E}_L) \setminus \{0\}$ and $\beta \in \mathrm{HF}_\bullet(H) \setminus \{0\}$ we have
$$c_\alpha(L)=c_{\Theta(\alpha)}(H), \quad c_\beta(H)=c_{\Psi(\beta)}(L).$$
\end{proposition}

\begin{proof}
Let $d>0$ be such that $\alpha \in \mathrm{im}(i^d)$. By properties of $\Theta$, we have the following commutative diagram with horizontal isomorphisms:
\begin{center}
	\begin{tikzcd}
	& \mathrm{HM}_\bullet(\mathcal{E}_L) \arrow{r}{\Theta} & \mathrm{HF}_\bullet(H) \\
	& \mathrm{HM}^d_\bullet(\mathcal{E}_L) \arrow{r}{\Theta} \arrow{u}{i^d} & \mathrm{HF}^d_\bullet(H) \arrow{u}{i^d}
	\end{tikzcd}
\end{center}
This implies that $\Theta(\alpha)$ lies in the image of the corresponding $i^d$ too. Taking the infimum of $d$'s with $\alpha \in \mathrm{im}(i^d)$ implies

$$c_{\Theta(\alpha)}(H) \leq c_{\alpha}(L).$$

The inequality above, however, is an equality: assume, by contradiction, that there is a value $e<c_\alpha(L)$ for which $\Theta(\alpha)$ lies in the image of $i^e$. Then the fact that $\Theta^{-1}$ is well defined and action preserving\footnote{One could also use $\Psi$ here, which actually is equal to $\Theta^{-1}$ on the homology level.} implies that $\alpha$ is in the image of $i^e$, contradicting the definition of $c_\alpha(L)$.

So we have shown 
$$c_{\Theta(\alpha)}(H)=c_\alpha(L).$$ The other equality can be shown analogously, or directly deduced from the first one by setting $\alpha=\Psi(\beta)$ and observing that $\Theta(\alpha)=\Theta(\Psi(\beta))=\beta$.
\end{proof}

Another result due to Abbondandolo and Majer \cite{am2006} asserts that the Morse homology of $\mathcal{E}_L$, with Lagrangians $L$ as above, is isomorphic to the free loop space homology, i.e.\ there exists an isomorphism
$$\Upsilon \, \colon H_\bullet(\mathcal{L}Q) \overset{\cong}{\longrightarrow} \mathrm{HM}_\bullet(\mathcal{E}_L).$$ This isomorphism also preserves the action filtration, i.e.\ it descends to yet another isomorphism:
$$\Upsilon \, \colon H_\bullet\left( \left\lbrace \mathcal{E}_L \leq d \right\rbrace\right) \overset{\cong}{\longrightarrow} \mathrm{HM}_\bullet^d(L),$$
for all $d \in \R$; cf.\ \cite[Section 2.4]{as2006}. Now recall that minimax values of $\mathcal{E}_L$ can be defined over singular homology classes as well: for $\alpha \in H_\bullet(\mathcal{L}Q) \setminus \{0\}$ we have
$$c_\alpha(L)=\inf_{\eta \in \alpha}\max \mathcal{E}_L \big |_{|\eta|},$$
where $\eta=\sum n_\sigma \sigma$ is a formal finite sum of simplices $\sigma \, \colon \Delta_\bullet \to \mathcal{L}Q$ and $|\eta|$ is the union of the image of those $\sigma$ with non-trivial coefficient $n_\sigma$. We adopt the notation
$$\mathcal{E}_L(\eta)\, :=\max \mathcal{E}_L \big |_{\vert \eta \vert}.$$ The same standard argument as in Proposition \ref{proposition floer and morse minimax agree} implies:
 
\begin{proposition}\label{proposition singular and morse agree}
If $\alpha \in H_\bullet(\mathcal{L}Q)\setminus \{0\}$, then
$$c_\alpha(L)=c_{\Upsilon(\alpha)}(L).$$

\end{proposition}

Once we have fixed $H$ and $L$, we can unambiguously talk about spectral invariants/minimax values associated to a singular/Morse homology class $\alpha$ thanks to Proposition \ref{proposition floer and morse minimax agree} and Proposition \ref{proposition singular and morse agree} and therefore we shall often simply write
$$c_\alpha(H)=c_\alpha(L)$$
without specifying $\alpha$. Of course, the Abbondandolo--Schwarz isomorphisms applied to $\alpha$ depend on the choice of $H$ and $L$.

\subsection{Stability and Spectrality}

We continue by establishing some expected properties of spectral invariants such as  $C^0$-Lipschitz continuity and spectrality. Both are slightly non-standard since we are in a non-compact setting. We start with the former.

\begin{proposition}\label{proposition spectral invariant is C0 lipschitz continuous}
Let $H^0,H^1$ be Hamiltonians satisfying (H0), (H1), (H2), and $$\Vert H^0-H^1 \Vert_{C^0}<\infty.$$ Then for every $\beta \in \mathrm{HF}_\bullet \setminus \{0\}$ we have
$$\vert c_\beta(H^0) - c_\beta(H^1) \vert \leq \Vert H^0 - H^1 \Vert_{C^0}.
$$
\end{proposition}

\begin{proof}
This follows verbatim from \cite[page 431]{schwarz2000} and/or \cite[page 2491]{irie2014}  --- the non-compactness does not affect the proof due to the imposed $C^0$-bound on the difference $H^0-H^1$. For convenience of the reader we go through the argument anyway. First of all observe
$$\Delta_{0,1}\colon=\int_0^1\sup_{(p,q) \in T^*Q}\left( H^0_t-H^1_t \right) \, dt= \Vert H^0 - H^1 \Vert_{C^0} <+\infty,$$ and by symmetry $\Delta_{1,0}<+\infty$. The $C^0$-bound also implies that the two Hamiltonians are close in the AS sense\footnote{See \cite[Lemma 1.21]{as2006} equation (1.39), (1.44) with $\varepsilon=0$.} and therefore there exists a direct continuation isomorphism
 $$\mathrm{HF}_\bullet(H^0) \overset{\cong}{\longrightarrow} \mathrm{HF}_\bullet(H^1).$$ This map is defined by the count of $s$-dependent Floer solutions $u \colon \R \to \mathcal{L}T^*Q$ connecting critical points $x \in \mathcal{P}(H^0)$ and $y \in \mathcal{P}(H^1)$. We claim that the continuation descends to the filtered version
$$\mathrm{HF}_\bullet^{<a}(H^0) \longrightarrow \mathrm{HF}_\bullet^{<a+\Delta_{0,1}}(H^1).$$
Indeed, for the usual connecting homotopy $H^s_t\colon=H^0_t+\beta(s) \cdot \left(H^1_t-H^0_t\right)$ with smooth
$$\beta \colon \R \longrightarrow \R, \quad \beta(s)=\begin{cases}
0, &\text{ if } s \leq 0, \\
1, &\text{ if } s \geq 1,
\end{cases}, \quad  \dot{\beta}(s) \geq 0,
$$
 we get
\begin{align*}
0 \leq E(u) &=\int_{\R} \Vert \partial_s u \Vert_J^2 \, ds \\
&=\int_\R -d\mathcal{A}_{H^s}[\partial_s u] \, ds \\
&=-\int_\R \frac{\partial}{\partial s} \left(\mathcal{A}_{H^s}(u(s)) \right) \, ds+\int_\R \frac{\partial \mathcal{A}_{H^s}}{\partial s}(u(s)) \, ds \\
&=\mathcal{A}_{H^0}(x)-\mathcal{A}_{H^1}(y)-\int_0^1 \dot{\beta}(s) \int_0^1\left(H^1_t(u(s))-H^0_t(u(s)) \right) \, dt \, ds \\
&=\mathcal{A}_{H^0}(x)-\mathcal{A}_{H^1}(y)+\int_0^1 \dot{\beta}(s) \int_0^1\left(H^0_t(u(s))-H^1_t(u(s)) \right) \, dt \, ds \\
&\leq \mathcal{A}_{H^0}(x)-\mathcal{A}_{H^1}(y) + \underbrace{\int_0^1 \sup \left( H^0_t-H_t^1\right) \, dt }_{=\Delta_{0,1}}.
\end{align*}
From the ``filtered" homomorphism above we can readily deduce that
$$c_\beta(H^1) \leq c_\beta(H^0)+\Delta_{0,1}.$$
Using the symmetry of the argument we also get
$$c_\beta(H^0) \leq c_\beta(H^1)+ \Delta_{1,0}$$
and with the final observation that 
$$\max \{\Delta_{0,1},\Delta_{1,0}\} \leq \Vert H^0-H^1 \Vert_{C^0(S^1 \times T^*Q)},$$
which finishes the proof.
\end{proof}

\begin{remark}
The $C^0$-bound in Proposition \ref{proposition spectral invariant is C0 lipschitz continuous} is implicitly used to ensure that the needed $L^\infty$-estimates on the $s$-dependent Floer solutions hold. If $\Vert H^0 -H^1 \Vert$ satisfies a certain quadratic bound \cite[Lemma 1.21]{as2006}, then the $s$-dependent Floer solutions lie in a compact set $B \subseteq T^*Q$. In particular, under the assumption of \cite[Lemma 1.21]{as2006} the conclusion of Proposition \ref{proposition spectral invariant is C0 lipschitz continuous} still holds after replacing $\Vert H^0-H^1 \Vert_{C^0}$ with $\Vert \left(H^0-H^1 \right) \big |_B \Vert_{C^0}$.

\end{remark}

Now we prove spectrality:

\begin{lemma}\label{lemma spectrality non-degenerate case}
Let $H$ be a Hamiltonian satisfying (H0), (H1), and (H2). Then the spectrum
$$\mathcal{S}(H)\, :=\left\lbrace \mathcal{A}_H(x) \, \big \mid \, x \in \mathcal{P}(H) \right\rbrace$$
is closed and discrete.
Moreover, for every $\alpha \in \mathrm{HF}_\bullet(H) \setminus \{0\}$ there exists $x \in \mathcal{P}(H)$ such that
$$c_\alpha(H)=\mathcal{A}_H(x) \in \mathcal{S}(H).$$
\end{lemma}

\begin{proof}

Let $(x_n)$ be a sequence of Hamiltonian orbits such that $\mathcal{A}_H(x_n)$ converges. In particular, $(x_n) \subseteq \mathcal{P}^d(H)$ for some $d \in \R$.  But \cite[Lemma 1.10]{as2006} tells us that $\mathcal{P}^d(H)$ is finite. This proves that $\mathcal{S}(H)$ is closed. To see that $\mathcal{S}(H)$ is discrete, observe that
$$(-\infty,d] \cap \mathcal{S}(H)=\mathcal{A}_H(\mathcal{P}^d(H)).$$

The proof of spectrality is the same as the one presented in \cite[Lemma 3.1]{irie2014}: assume by contradiction that
$$c_\alpha(H) \notin \mathcal{S}(H).$$
First of all, observe that $-\infty < c_\alpha(H)$. If not there exists a sequence $(x_n) \subset \mathcal{P}(H)$  with $\mathcal{A}_H(x_n) \to -\infty$, contradicting finiteness of $\mathcal{P}^d(H)$. Since $\mathcal{S}(H)$ is closed (as seen above), there exists a $\varepsilon>0$ such that 
$$[c_\alpha(H)-\varepsilon,c_\alpha(H)+\varepsilon) \cap \mathcal{S}(H)=\emptyset.$$ 
In particular
$$\mathrm{HF}_\bullet^{[c_\alpha(H)-\varepsilon,c_\alpha(H)+\varepsilon)}(H)=0,$$ and hence the homomorphism
$$\mathrm{HF}_\bullet^{c_\alpha(H)-\varepsilon}(H) \longrightarrow \mathrm{HF}_\bullet^{c_\alpha(H)+\varepsilon}(H)$$
is an isomorphism. This however readily contradicts the infimum definition of $c_\alpha(H)$. Therefore we have shown that
$$c_\alpha(H) \in \mathcal{S}(H).$$
\end{proof}

The $x \in \mathcal{P}(H)$ produced by Lemma \ref{lemma spectrality non-degenerate case} will be referred to as a \textbf{carrier} of the spectral invariant $c_\alpha(H)$. Note that a carrier does not have to be unique. Moreover, a carrier inherits the homotopy class of the corresponding connected component of $\mathcal{L}Q$ on which $\alpha$ is defined. More precisely, if the Floer class $\alpha$ corresponds to a singular homology class on the component $\mathcal{L}_\eta Q$, where $\eta$ is a conjugacy class in $\pi_1(Q)$, then the free homotopy class of the carrier $x$ is $\eta$. Indeed, the AS isomorphism respects the homotopy class of the generators, i.e.\
$$H_\bullet(\mathcal{L}_\eta Q) \cong \mathrm{HM}_\bullet\left( \mathcal{E}_L \big |_{\mathcal{L}_\eta Q}\right) \cong \mathrm{HF}_\bullet(H;\eta),$$
where the latter denotes the Floer homology of $H$ generated by Hamiltonian orbits with fixed free homotopy class $\eta$.

Until now we have only considered $1$-periodic Hamiltonians $H$ satisfying all three AS conditions. However, we would like to work with spectral invariants for \emph{autonomous} Hamiltonians which at most satisfy (H1) and (H2), e.g. $F$ in Section \ref{section Reeb Orbits on Starshaped Domains}. There are multiple ways to deal with this and we opt for the following: let $H$ be an autonomous Hamiltonian satisfying (H1) and (H2) and consider $$H_n \, \colon  S^1 \times T^*Q \longrightarrow \R, \; H_n(t,q,p)=H(q,p)+ W_n(t,q),$$ a sequence of Hamiltonians satisfying (H0), (H1), and (H2), where $W_n \, \colon S^1 \times Q \to \R$ are potentials tending to $0$ in $C^2$; cf.\ \cite{weber2002}.

\begin{definition}\label{definition spectral invariant for degenerate Hamiltonians}
Let $\alpha \in H_\bullet(\mathcal{L}Q) \setminus \{0\}$ and $H \, \colon T^*Q \to \R$ a Hamiltonians satisfying (H1) and (H2). Define
$$c_\alpha(H)\, :=\lim_{n\to \infty}c_\alpha(H_n),$$
where $H_n=H+W_n$ are non-degenerate Hamiltonians with $1$-periodic potentials $W_n$ that tend to $0$ in $C^2$.
\end{definition}

With Proposition \ref{proposition spectral invariant is C0 lipschitz continuous} at hand we can show that Definition \ref{definition spectral invariant for degenerate Hamiltonians} works and that spectrality also holds in the degenerate case:

\begin{proposition}\label{proposition spectral invariant for degenerate Hamiltonians }
Let $\alpha$ and $H$ as in Definition \ref{definition spectral invariant for degenerate Hamiltonians}. Then $c_\alpha(H)$ is well defined and there exists $x \in \mathcal{P}(H)$ such that $$c_\alpha(H)=\mathcal{A}_H(x) \in \mathcal{S}(H).$$
\end{proposition}

\begin{proof}
The proof strategy is a refinement of standard arguments to the non-compact setting \cite[Proposition 5.1]{fs2007}. Let $H_n$ be a sequence of non-degenerate Hamiltonians converging to $H$ as in Definition \ref{definition spectral invariant for degenerate Hamiltonians}. From Lemma \ref{lemma spectrality non-degenerate case} we get the existence of a sequence $x_n \in \mathcal{P}(H_n)$ such that $$c_\alpha(H_n)=\mathcal{A}_{H_n}(x_n).$$
Additionally, a thorough inspection of \cite[Lemma 1.10]{as2006} reveals that the constants therein depend continuously on $H_n$, thus showing that $(x_n)$ lies in a compact set $B$ in $T^*Q$ --- see also  \cite[Theorem 8.8]{brezis2011}. Hence, by means of Arzel\`a-Ascoli,\footnote{For equicontinuity observe $$d(x_n(t_0),x_n(t_1)) \leq \int_{t_0}^{t_1} \Vert \dot{x}_n(t) \Vert \, dt \leq \Vert X_{H_n} |_B \Vert \cdot \vert t_0-t_1 \vert \leq \underbrace{\Vert H_n \Vert_{C^1(B)}}_{\text{bounded in $n$}} \cdot \vert t_0-t_1 \vert.$$} the $x_n$ converge to some closed loop $x$ in $C^0$. The $C^2$-convergence of the Hamiltonians and bootstrapping imply that $x$ is a smooth $1$-periodic orbit of $H$ with $\mathcal{A}_H(x)=\lim_{n \to \infty}c_\alpha(H_n)$. This limit is unique since for any other sequence $\tilde{H}_n=H+V_n$ we can apply Proposition \ref{proposition spectral invariant is C0 lipschitz continuous}:
$$\vert c_\alpha(\tilde{H}_n)-c_\alpha(H_n) \vert \leq \Vert \tilde{H}_n-H_n \Vert_{C^0}=\Vert V_n - W_n \Vert_{C^0} \to 0, \text{ for } n \to \infty.$$
\end{proof}

\begin{remark}
Since the carrier $x$ of $c_\alpha(H)$ in Proposition \ref{proposition spectral invariant for degenerate Hamiltonians } is constructed as a limit of carriers $x_n$, the homotopy property still holds for degenerate $H$, i.e.\ if $\alpha$ is a singular homology class on $\mathcal{L}_\eta Q$ and $x$ is a carrier of $c_\alpha(H)$, then $[x]=\eta$.
\end{remark}

For degenerate Lagrangians we can exploit the results on the dual Hamiltonian side to recover the fact that the minimax values are attained: let $L \, \colon TQ \to \R$ be a Lagrangian in the sense of AS modulo the non-degeneracy condition. Denote by $H$ its dual Hamiltonian, which then satisfies (H1) and (H2), but is degenerate as well. Subtract a small potential $W_n \, \colon S^1 \times Q \to \R$ from $L$ so that the resulting Lagrangian  $L_n(t,q,v)\, :=L(q,v)-W_n(t,q,v)$ is non-degenerate. These Lagrangians admit dual Hamiltonians $H_n$, which are of the form $H_n=H+W_n$ and satisfy (H0), (H1), and (H2). In particular, using Proposition \ref{proposition floer and morse minimax agree} and Proposition \ref{proposition spectral invariant for degenerate Hamiltonians } we get a unique limit
$$\lim_{n \to \infty} c_\alpha(L_n)=\lim_{n \to \infty} c_\alpha(H_n)=c_\alpha(H).$$

If the gradient flow of $\mathcal{E}_L$ is already sufficiently nice, e.g. whenever $L$ is a purely kinetic Lagrangian, Lusternik--Schnirelmann theory (LS theory) is applicable and produces minimax values $c_\alpha(L)$ that are attained as critical values of $\mathcal{E}_L$ \cite[Chapter 2]{klingenberg2012}. We show that the limit of the sequence $c_\alpha(L_n)$ above produces the same value as LS theory for $c_\alpha(L)$. This is crucial in order to relate the spectral invariants to the minimax values, and we will apply this to the kinetic Lagrangian action functional in the proof of the Main Theorem \ref{theorem main section existence}.

\begin{proposition}\label{proposition LS spectral invariant and limit for degenerate Lagrangians agree}
Let $L$ and $L_n=L-W_n$ be as above and assume that the gradient flow of $\mathcal{E}_L$ is subject to LS theory. Let $\alpha \in H_\bullet(\mathcal{L}Q) \setminus \{0\}$. Then the limit of the sequence $c_\alpha(L_n)$ agrees with the minimax value $c_\alpha(L)$ produced by LS theory, i.e.\
$$c_\alpha(L)=\lim_{n \to \infty}c_\alpha(L_n).$$
\end{proposition}

\begin{proof}
First of all observe that $c_\alpha(L_n)$ can also be computed by viewing $\alpha$ as a singular homology class; cf.\ Proposition \ref{proposition singular and morse agree}. Hence
$$\big \vert c_\alpha(L_n)-c_\alpha(L) \big \vert \leq \bigg \vert \inf_{\tau \in \alpha} \sup \mathcal{E}_{L_n} \big |_{\vert \tau \vert}-\inf_{\eta \in \alpha} \sup \mathcal{E}_L \big |_{\vert \eta \vert} \bigg \vert
= \vert \mathcal{E}_{L_n}(q_n)-\mathcal{E}_L(q) \vert,
$$
for some $q \in \mathrm{Crit}(L)$ and $q_n \in \mathrm{Crit}(L_n)$ --- such $q_n$'s  (resp. $q$) exist because $$c_\alpha(L_n)=c_\alpha(H_n)=\mathcal{A}_{H}(x_n)=\mathcal{E}_{L_n}(q_n),$$ by Proposition \ref{proposition floer and morse minimax agree} and Lemma \ref{lemma spectrality non-degenerate case} (resp. standard LS theory). 

We proceed with a case distinction first consider the case where there exists a subsequence, still denoted by $\mathcal{E}_{L_n}(q_n)$, such that $$\mathcal{E}_L(q) \leq \mathcal{E}_{L_n}(q_n).$$ We claim that 
$$\big \vert c_\alpha(L_n)-c_\alpha(L) \big \vert \leq \Vert W_n \Vert_{C^0} \to 0.$$
The LHS is equal to $\mathcal{E}_{L_n}(q_n)-\mathcal{E}_L(q)$. Denote by $\eta$ the singular chain with $\mathcal{E}_L(\eta)=\mathcal{E}_L(q)$. Since $\eta$ represents $\alpha$, the definition of $c_\alpha(L_n)$ and the identity $L_n=L-W_n$ imply $$\mathcal{E}_{L_n}(q_n)=c_\alpha(L_n) \leq \mathcal{E}_{L_n}(\eta) \leq \mathcal{E}_L(\eta) + \Vert W_n \Vert_{C^0}= \mathcal{E}_L(q)+\Vert W_n \Vert_{C^0}.\footnote{The supremum of a sum is smaller than the sum of the individual suprema.}$$
For the other case, denote by $\tau_n$ a representative of $\alpha$ with $\mathcal{E}_{L_n}(\tau_n)=\mathcal{E}_{L_n}(q_n)$ and use the same logic to deduce $$\mathcal{E}_L(q)=c_\alpha(L) \leq \mathcal{E}_L(\tau_n) \leq \mathcal{E}_{L_n}(\tau_n)+\Vert W_n \Vert_{C^0}=\mathcal{E}_{L_n}(q_n)+\Vert W_n \Vert_{C^0}.$$ Combining the two cases grants
$$\big \vert c_\alpha(L_n)-c_\alpha(L) \big \vert \leq \Vert W_n \Vert_{C^0} \to 0.$$
This concludes the proof.
\end{proof}

\section{Pinching and Floer Homologies}\label{section Pinching}

\subsection{Preliminaries}
We closely follow \cite{ms2011, heistercamp2011, wullschleger2014} and construct three sequences of non-degenerate Hamiltonians. Let $F \, \colon T^*Q \to \R$ be the Hamiltonian that realizes $\Sigma$ as in Section \ref{section Reeb Orbits on Starshaped Domains}. In particular, $\Sigma$ is \emph{not} assumed to be non-degenerate throughout this section.\footnote{The non-degeneracy of the relevant data to define Floer homology will be achieved by adding small time-dependent potentials as we will see shortly.} Up to rescaling the Riemannian metric we can assume that the kinetic Hamiltonian $$G(q,p)=\frac{1}{2} \Vert p \Vert^2$$ is pointwise smaller than $F$, i.e.\ $$G \leq F.$$ Since $Q$ is compact, there exists a constant $\sigma >1$ such that $$\sigma G \geq F.$$ Define a smooth auxiliary function

$$
f(r)=\begin{cases}
0, \quad &r \in (-\infty,\varepsilon^2], \\
r, \quad & r \geq \varepsilon,
\end{cases}
$$
with $0 \leq f'(r) \leq 2$. The inequality
$$f \circ G \leq f \circ F \leq \sigma G$$
still holds.

We will need the three Hamiltonians to agree at infinity to exploit the full power of the pinching. For this purpose we define, for fixed $d \geq 0$, a function $$\tau_d \, \colon \R \longrightarrow \R, \quad
\tau_d(r)=
\begin{cases}
0, \quad & r \in (-\infty,\sqrt{2d}], \\
1, &r \in [2\sqrt{2d},+\infty),
\end{cases}
$$
with $\tau_d' \geq0$. The square root comes from the fact that we will feed $\Vert p \Vert$ to $\tau_d$ and use the following relations $$\sqrt{2d} = \Vert p \Vert \iff d=\frac{1}{2} \Vert p \Vert^2 \iff d=G(q,p).$$ Let 
$$W_{n} \, \colon S^1 \times Q \longrightarrow \R$$ be a sequence of $1$-periodic potentials, $n \in \N$. Pick $$c,c_{n} \in (0,1/4) \text{ with } c \geq c_{n} \text{ and } \lim_{n \to \infty}c_n=0,$$
and define
$$c'_{n}=\min \left\lbrace c_{n},\frac{c_{n}}{\Vert X_{f \circ F} \Vert_{C^0}}, \frac{c_{n}}{\Vert X_{\sigma G} \Vert_{C^0}},\frac{c_{n}}{\Vert X_{f \circ G} \Vert_{C^0}} \right\rbrace.$$ Assume
$$\Vert W_{n} \Vert_{C^1}  < c'_{n}.$$
In particular, by choice of $c_n$, we obtain
$$\Vert W_{n} \Vert_{C^0} \to 0 \text{ as } n \to \infty.$$
We finally define the 3 Hamiltonians and suppress the dependence on $d>0$ in the notation:

\begin{align*}
G_{n}^+(t,q,p)&=\sigma \cdot G(q,p) + W_{n}(t,q), \\
K_{n}(t,q,p)&=(1-\tau_{c+d})(\Vert p \Vert) \cdot \left[f \circ F(q,p)+W_{n}(t,q) \right]+\tau_{c+d}(\Vert p \Vert) G_{n}^+(t,q,p), \\
&=(1-\tau_{c+d})(\Vert p \Vert) \cdot \left(f \circ F(q,p)\right)+\tau_{c+d}(\Vert p \Vert)\sigma G(t,q,p)+ W_{n}(t,q), \\
G_{n}^-(t,q,p)&=(1-\tau_{c+d})(\Vert p \Vert) \cdot \left[f \circ G(q,p)+W_{n}(t,q) \right]+\tau_{c+d}(\Vert p \Vert) G_{n}^+(t,q,p), \\
&=(1-\tau_{c+d})(\Vert p \Vert) \cdot \left(f \circ G(q,p)\right)+\tau_{c+d}(\Vert p \Vert)\sigma G(t,q,p)+ W_{n}(t,q).
\end{align*}
By definition, $G_{n}^+, \, K_{n}$ and $G_{n}^-$ agree on $\{G \geq 4d\}$. It is well known that for each Hamiltonian $H \in \{f \circ G, \, f \circ F, \,\sigma G \}$, the set of potentials $$W \in \mathcal{V}_H \subseteq C^{\infty}(S^1 \times Q)$$ such that $H +W$ is non-degenerate, is open and dense. The intersection of these three sets of potentials is a residual set, and thus dense by the Baire Category Theorem. This allows us to pick the sequence $W_n$ such that the above estimates are satisfied and all Hamiltonians $H+W_n$ are non-degenerate; cf.\ \cite{weber2002}.
Towards the end of the present section we will pass to the degenerate case, i.e.\ send $W_{n}$ to $0$ for $n \to \infty$. We introduce the needed notation in advance:

\begin{align*}
K&=(1-\tau_{c+d}) \cdot (f \circ F)+ \tau_{c+d} \cdot \sigma G, \\
G^-&=(1-\tau_{c+d}) \cdot (f \circ G)+\tau_{c+d} \cdot \sigma G.
\end{align*}
The following is a collection of standard results taken from \cite{heistercamp2011}, which will be used throughout the whole subsection.

\begin{lemma}\label{lemma summary Heistercamp}
Let $H \, \colon T^*Q \to \R$ be a $C^1$-map which is homogeneous of degree $2$ on every fibre and $c>0$.\footnote{Any $c>0$, not necessarily the one chosen previously.} Let $V \, \colon S^1 \times T^*Q \to \R$ smooth with
$$\Vert V \Vert_{C^1} < \frac{c}{\Vert X_H \Vert_{C^0}}.$$ Let $x \in \mathcal{P}(H+V)$, $a=H(x(t_0))$ for some fixed $t_0 \in S^1$. Then
\begin{enumerate}[(i)]
\item $a-c < H(x(t)) < a+c, \quad \forall t \in S^1.$
\end{enumerate}

Let $W \, \colon S^1 \times Q \to \R$ be smooth with $\Vert W \Vert_{C^1}<c$, $h \, \colon \R \to \R$ smooth and $r>0$. Then for $x \in \mathcal{P}(h \circ H+W)$:

\begin{enumerate}[(i)]
\setcounter{enumi}{1}
\item $\mathcal{A}_{h\circ H +W}(x)=\int_0^1 2h'(H(x(t)) \cdot H(x(t)) - h(H(x(t))-W_t(x(t)) \, dt$,
\item  $\mathcal{S}(rH+W) \subseteq \frac{1}{r} \mathcal{S}(H+W)+[-c,c]$.
\end{enumerate}

\end{lemma}

The following is an adaptation of \cite[Proposition 2.2.2]{heistercamp2011}.

\begin{proposition}\label{proposition control of gamma}
Let $x \in \mathcal{P}(K_{n})$, $d > 4 \varepsilon + c$. Then
\begin{itemize}
\item If there is a $t_0$ with $F(x(t_0))>d$, then $\mathcal{A}_{K_{n}}(x) > d-2c$,
\item  if $F(x(t_0)) \leq d$ instead, then $\mathcal{A}_{K_{n}}(x) \leq d+c$.
\end{itemize}
The same result remains true after swapping $K_n$ and $F$ with $G^-_n$ and $G$.
\end{proposition}

\begin{proof}
In the first situation we get $d <F(x(t_0))=f \circ F(x(t_0))$ because $d > \varepsilon$ and the definition of $f$. In particular, using $\sigma G \geq f \circ F$:
$$a\colon=K_{n}(x(t_0))-W_{n}(x(t_0)) \geq f \circ F(x(t_0)) >d.$$ Thus item $(i)$ in Lemma \ref{lemma summary Heistercamp} applied to $H\colon=K_{n}-W_{n}$, $V=W_{n}$, and the choice of $c$, tells us
$$K_{n}(x)-W_{n}(x) > a-c > d-c.$$ At the same time we have $ F(x) >  \varepsilon$ --- indeed, if we assume $F(x(s)) \leq \varepsilon$ for some $s \in S^1$, we reach a contradiction: the inequality $G(x(s)) \leq F(x(s)) \leq \varepsilon <d$ implies $$K_n(x(s))=f \circ F(x(s))+W_n(x(s))$$ and therefore
$$ \varepsilon \geq F(x(s)) \geq f \circ F(x(s))=K_{n}(x(s))-W_{n}(x(s)) >d-c,$$
which contradicts the choice of $\varepsilon$.

Since $F(x) \geq \varepsilon$, we have $f \circ F(x)=F(x)$, in particular 

\begin{align*}
d(f \circ F)(x)[Y]&=f'(F(x)) \cdot dF(x)[Y]=dF(x)[Y].
\end{align*} We will need this in a second. Observe

$$\mathcal{A}_{K_{n}}(x)=\int x^*\lambda-\int_0^1 K_{n}(x) \, dt=\int_0^1 dK_{n}(x)[Y]-K_{n}(x) \, dt.$$
But
\begin{align*}
dK_{n}(x)[Y]&=\underbrace{-\Vert p \Vert \cdot (\tau_{d+c})'(\Vert p \Vert)(f \circ F)(x)+\Vert p \Vert \cdot (\tau_{d+c})'(\Vert p \Vert) \sigma G(x)}_{ \geq 0} \\
& \quad + (1-\tau_{d+c})(\Vert p \Vert) \cdot d(f \circ F)(x)[Y]+ \tau_{d+c}(\Vert p \Vert) \cdot \sigma dG(x)[Y]+dW_{n}(x)[Y] \\
& \geq 2 (1-\tau_{d+c})(\Vert p \Vert) F(x) + \tau_{d+c}(\Vert p \Vert) 2 \sigma G(x) \\
&= 2 \cdot K_{n}(x)-2 W_{n}(x)
\end{align*}
Here we used the previously derived identity $d(f \circ F)(x)[Y]=2 F(x)$, the analogous Euler identity $dG(x)[Y]=2G(x)$, and $dW_{n}(x)[Y]=0$ (note that $W_{n}$ does not depend on the fibre variable $p$). Thus we can finally bound

\begin{align*}
\mathcal{A}_{K_{n}}(x)&=\int dK_{n}(x)[Y]-K_{n}(x) \, dt \\
& \geq \int \left( K_{n}(x)-W_{n}(x) \right) - W_{n}(x) \, dt \\
& \geq d-c -c \\
&=d-2c.
\end{align*}
This concludes the proof of the first bullet point.

For the other item we observe that $F(x(t_0)) \leq d$ implies $G(x(t_0)) \leq d$, thus $\tau_{c+d}$ does vanish at $\Vert p(t_0) \Vert$. In particular
$$a=K_{n}(x(t_0))-W_{n}(x(t_0))=f \circ F(x(t_0)) \leq F(x(t_0)) \leq d.$$
Applying $(i)$ from Lemma \ref{lemma summary Heistercamp} again tells us that 
$$ K_{n}(x)-W_{n}(x) < a+c \leq d+c,$$
but $$f \circ F \leq K_{n} - W_{n},$$
therefore implying with the above that $G(x) \leq d+c$ and thus that $\tau_{c+d}$ vanishes on the whole interval $\Vert p(t) \Vert$. This proves
$$K_{n}=f \circ F(x)+W_{n}(x).$$ Invoking Lemma \ref{lemma summary Heistercamp} item $(ii)$ gives 
$$\mathcal{A}_{K_{n}}(x)=\mathcal{A}_{f \circ F+W_{n}}(x)=\int 2 f'(F(x))-f(F(x)) \, dt - \int W_{n}(x) \, dt.$$
We make a case distinction to bound the left integral: if $F(x(s)) < \varepsilon^2$, then $$2 f'(F(x(s)))-f(F(x(s)))=0.$$ For $F(x(s)) > \varepsilon$ we have $$2 f'(F(x(s)))-f(F(x(s)))=2F(x(s))-F(x(s)) \leq d+c,$$ and in the case $F(x(s)) \in [\varepsilon^2,\varepsilon]$ we get $$2 f'(F(x(s)))-f(F(x(s))) \leq 4F(x(s)) < 4\varepsilon \leq d$$ since $f' \leq 2$ and by assumption on $d$. All in all this proves
$$\mathcal{A}_{K_{n}}(x) \leq d+c.$$

The case for $G^-_{n}$ is verbatim the same after swapping $F$ with $G$.
\end{proof}

\begin{notation}
We write $\underline{\mathcal{P}}(H)$ to denote the non-constant orbits of the Hamiltonian $H$. The notation for action windows $\underline{\mathcal{P}}^b$ is adopted.
\end{notation}

\begin{proposition}\label{proposition precomposing with f does not change orbits}
Let $d>\max\{4 \varepsilon+c,2c\}$ and $f, \, K, \, G^-$ be as above. Then for any $$b \in (0,d-2c)$$ we have
$$\underline{\mathcal{P}}^b(K) \subseteq \underline{\mathcal{P}}^b(f \circ F)=\underline{\mathcal{P}}^b(F)$$
and
$$\underline{\mathcal{P}}^b(G^-) \subseteq \underline{\mathcal{P}}^b(f \circ G)=\underline{\mathcal{P}}^b(G).$$

\end{proposition}

\begin{proof}
The function $f$ does implicitly depend on $\varepsilon$. We may assume that $\varepsilon>0$ has been chosen so small that the non-constant $1$-periodic orbits of $f \circ G$ (resp. $f \circ F$) agree with those of $G$ (resp. $F$) and lie in the set $\{f \circ G \geq \varepsilon\}$ (resp. $\{f \circ F \geq \varepsilon \}$); cf.\ \cite[page 52]{heistercamp2011}. In particular, for any $x \in \underline{\mathcal{P}}(f \circ F)$ we have $$\mathcal{A}_{f \circ F}(x)=f \circ F(x)=F(x)=\mathcal{A}_F(x),$$ with the analogous statement for $G$. Therefore we have just shown that for any $b>0$ we have $$\underline{\mathcal{P}}^b(f \circ G)= \underline{\mathcal{P}}^b(G) \text{ and } \underline{\mathcal{P}}^b(f \circ F)=\underline{\mathcal{P}}^b(F).$$

Now $x \in \mathcal{P}^b(G^-)$ with $b$ chosen as in the statement. In particular, $\mathcal{A}_{G^-}(x) \leq b<d-2c$ and thus the contrapositive of Proposition \ref{proposition control of gamma} applied to $G^-$ (i.e.\ $G_n^-$ with $W_n=0$) says $G(x) \leq d$. Therefore $G^-(x)=f \circ G$, hence
$$\mathcal{P}^b(G^-) \subseteq \mathcal{P}^b(f \circ G).$$ The analogous inclusion holds for $K$ and $F$. Restricting to the set of non-constant orbits and combining this with the above then concludes the proof.
\end{proof}

We are finally able to state and prove our own version of the``Non-Crossing Lemma"; cf.\ \cite[Lemma 3.3]{ms2011}  and \cite[Lemma 2.2.3]{heistercamp2011}. This will enable us to get a filtered continuation isomorphism between the Floer homologies of $G^-_{n}$ and $G^{+}_{n}$. To this end, define the usual homotopy
$$G^s_{n}\, :=(1-\beta(s)) \, G^-_{n}+\beta(s) \, G^+_{n}$$ with $\beta$ chosen as in the proof of Proposition \ref{proposition spectral invariant is C0 lipschitz continuous}. We also define, for any $a \in \R$:
$$a(s)\, :=\frac{a}{1 + \beta(s)\left(\sigma-1\right)}.$$
The function $a(s)$ is strictly decreasing, starts at $a(0)=a$ and ends at $a(1)=\frac{a}{\sigma}$. In particular
$$\frac{a(0)}{a(1)}=\sigma.$$

\begin{lemma}[Non-Crossing Lemma]\label{lemma non-crossing}
Fix $d >\max\{4\varepsilon+c,2c\}$ as in Proposition \ref{proposition precomposing with f does not change orbits}. Up to shrinking $\varepsilon=\varepsilon(c,\sigma,G)>0$, the following holds true:

For any $a \in \R$ with $$4(\sigma+4)c<a<d-2c$$ the relation $$[a-c_{n},a+c_{n}] \cap \mathcal{S}(G+W_{n}) = \emptyset$$ implies
$$a(s) \notin \mathcal{S}(G_{n}^s), \quad \forall s \in [0,1].$$
\end{lemma}

\begin{proof}
We will make a case distinction and show that for every $x \in \mathcal{P}(G^s_{n})$ we have:
\begin{enumerate}
\item if $x$ enters  $\{ G>d\}$, then $\mathcal{A}_{G_{n}^s}(x) > a(s)$, 
\item if $x$ enters $\{ G \leq \varepsilon\}$, then $\mathcal{A}_{G^s_{n}}(x)<a(s)$ (if $\varepsilon$ is chosen accordingly),
\item if $x$ stays in $\{ \varepsilon < G \leq d\}$, then $[a-c_{n},a+c_{n}]$ intersects $\mathcal{S}(G+W_{n})$, contradicting our assumption.
\end{enumerate}
Showing $1, \, 2$ and $3$ above immediately implies our desired result. 

Ad $1$: the same computation as in Proposition \ref{proposition control of gamma} implies
$$\mathcal{A}_{G_{n}^s}(x) > d-2c>a \geq a(s).$$ This proves $1.$

Ad $2$: By part 1. we may assume that $x$ stays in $\{G \leq d\}$, in particular we can ignore $\tau_{c+d}$ and deduce
$$G_{n}^s(x)=(1-\beta(s)) (f \circ G)(x)+\beta(s) \cdot \sigma G(x)+W_{n}(x).$$
This readily implies
$$f \circ G(x)  \leq G^s_{n}(x)-W_{n}(x).$$ For simplicity, define $$\bar{h}(r):=\beta(s)\sigma \cdot r +(1-\beta(s)) \cdot f(r)$$ 
and rewrite the above as
$$\bar{h}\circ G(x)+W_{n}(x)=G^s_{n}(x).$$ 
By the assumption in 2 there is some $t_0$ with $ G(x(t_0)) \leq \varepsilon$. We apply Lemma \ref{lemma summary Heistercamp} $(i)$ to

$$G^s_{n}(x(t_0))-W_{n}(x(t_0)) =\bar{h}(G(x(t_0))$$
and use the choice of $c_{n}' > \Vert W_{n} \Vert_{C^1}$ to obtain

\begin{align*}
\bar{h}(G(x))&\leq \bar{h}(G(x(t_0)) + \Vert W_{n} \Vert_{C^1} \cdot \left((1-\beta(s)) \Vert X_{g_i \circ G} \Vert_{C^0} +\beta(s) \Vert X_{\sigma G} \Vert_{C^0} \right) \\
 &\leq \sigma \varepsilon + \varepsilon +2c_{n} \\
 &=(\sigma+1) \varepsilon + 2c_{n,}.
\end{align*}

We claim that
$$G(x) \leq (\sigma+1 ) \varepsilon + 2c_{n}$$
holds. Indeed, if not, then there is some $r_0 \in S^1$ with $ G(x(r_0))>(\sigma+1 ) \varepsilon + 2c_{n}$. In particular, $ G(x(r_0))> \varepsilon$ which implies $f \circ G(x(r_0))= G(x(r_0))$. However, as seen above: $$f \circ G(x) \leq G_{n}^s(x)-W_{n}(x)=\bar{h}(G(x)) \leq (\sigma+1)\varepsilon+2c_{n},$$ which altogether leads to the following contradiction:
$$(\sigma+1) \varepsilon + 2c_{n}<  G(x(r_0))=f \circ G(x(r_0)) \leq G_{n}^s(x(r_0))-W_{n}(x(r_0)) \leq (\sigma+1) \varepsilon + 2c_{n}.$$
The claim follows.

Next we bound the action $\mathcal{A}_{G_{n}^s}$:
Observe $\bar{h}'(r)=\beta(s)\sigma +(1-\beta(s)) \cdot f'(r)$. Applying Lemma \ref{lemma summary Heistercamp} item $(ii)$ to $\bar{h} \circ G+W_{n}$ gives

\begin{align*}
\mathcal{A}_{G_{n}^s}(x)&=\int 2 \left[ \sigma\beta(s)+(1-\beta(s)) \cdot f'(G(x))\right] \cdot G(x) - \underbrace{(\bar{h}(G(x)+W_{n}(x))}_{=G_{n}^s(x) \geq -c_{n}} \, dt \\
&\leq \int (2\sigma +4)G(x)\, dt +c_{n}.
\end{align*}
Together with the claim from above we obtain

$$
\mathcal{A}_{G^s_{n}}(x) \leq 2(\sigma+1)^2\varepsilon+4(\sigma +2)c_{n}+c_{n} \leq 2(\sigma + 1)^2 \varepsilon + 4\cdot(\sigma +3)c.
$$
We choose $\varepsilon$ now: since $a / \sigma \leq a(s)$, it suffices to have $$0< \varepsilon < \frac{\frac{a}{\sigma}-4(\sigma+3)c}{2(\sigma+1)^2}$$ to obtain $\mathcal{A}_{G_{n}^s}(x)<a(s)$. By choice of $a$ we have $a > 4(\sigma+4)c$, so we might as well choose 
$$0< \varepsilon(c,\sigma,G) < \frac{2c}{(\sigma+1)^2}.$$ This takes care of case 2.

Ad 3: By contradiction assume that that $\mathcal{A}_{G^s_{n}}(x)=a(s)$ for some $s \in [0,1]$. Since $x$ lies in $\{ \varepsilon< G \leq d\}$ we have 
\begin{align*}
G_{n}^s(x)&=(1-\beta(s)) \cdot G(x)+\beta(s) \sigma \cdot G(x)+W_{n}(x) \\
&=\left[(\sigma - 1) \beta(s)+1 \right]\cdot G(x)+W_{n}(x).
\end{align*}
In particular $a(s) \in \mathcal{S}\left(\left[(\sigma - 1) \beta(s)+1 \right]\cdot G+W_{n}\right)$, thus by Lemma \ref{lemma summary Heistercamp} $(iii)$:
$$[a-c_{n},a+c_{n}] \cap \mathcal{S}(G+W_{n}) \neq \emptyset,$$
contradicting our assumption. This finishes the proof.
\end{proof}

\begin{notation}
For better readability we set 
$$e=4(\sigma+4)c.$$
\end{notation}

The condition in the Non-Crossing Lemma is met for almost all $n \in \N$ as the next result shows: 

\begin{proposition}\label{proposition non-crossing assumptions hold asymptotically}
Let $a \notin \mathcal{S}(G)$. Then there exists a positive integer $N=N(a)$ such that:
$$\forall n \geq N(a) \, \colon \quad [a-c_{n},a+c_{n}] \cap \mathcal{S}(G+W_{n}) = \emptyset.$$
\end{proposition}

\begin{proof}
Assume by contradiction that for all positive integers $N$ there exists a $n \geq N$ such that $[a-c_n,a+c_n] \cap \mathcal{S}(G+W_n) \neq \emptyset$. In particular, we can extract a strictly increasing subsequence $(n_k)_{k \in \N}$ such that $[a-c_{n_k},a+c_{n_k}] \cap \mathcal{S}(G+W_{n_k}) \neq \emptyset.$ Any fixed sequence $$a_k \in [a-c_{n_k},a+c_{n_k}] \cap \mathcal{S}(G+W_{n_k})$$ is bounded since $0<c_n \leq c$ and thus admits a convergent subsequence, still denoted by $a_k$. But $c_{n_k} \to 0$ for $k \to \infty$, which then forces $a_k$ to converge to $a$ with $$a \in \mathcal{S}(G),$$ since $W_{n_k} \to 0$. This contradicts the assumption on $a$.
\end{proof}

Now we extract a whole interval $$[a,b] \subseteq (e,d-2c]$$ for which Proposition \ref{proposition non-crossing assumptions hold asymptotically} then holds. Pick any $b \in (e,d-2c]$ with $b \notin \mathcal{S}(G)$, and apply Proposition \ref{proposition non-crossing assumptions hold asymptotically} to get a $N=N(b)$. Up to increasing $N$ we can assume $e< b-c_{n}$ for all $n \geq N$. We pick $a \in [b-c_{N},b)$ with $a \notin \mathcal{S}(G)$. In particular $e < a$, thus $$[a,b]\subseteq (e,d-2c].$$
\begin{claim no number}
There exists $N(a,b) \geq N(b)$ such that 
$$\forall n \geq N(a,b) \, \colon \quad [a-c_{n},b+c_{n}] \cap \mathcal{S}(G+W_{n}) = \emptyset.$$
\end{claim no number}
Indeed, since $a \notin \mathcal{S}(G)$, there exists a $N(a)$ such that $[a-c_{n},a+c_{n}]$ does not intersect $\mathcal{S}(G+W_{n})$ for all $n \geq N(a)$; cf.\ Proposition \ref{proposition non-crossing assumptions hold asymptotically}. Taking $N(a,b)=\max\{N(a),N(b)\}$ and $n \geq N(a,b)$ we get
$$\emptyset=\left([a-c_{n},a+c_{n}] \cup [b-c_{n},b+c_{n}]\right) \cap \mathcal{S}(G+W_{n})=[a-c_{n},b+c_{n}] \cap \mathcal{S}(G+W_{n}),$$
which proves the claim.

\begin{lemma}\label{lemma adiabatic argument}
Let $$b \in (e,d-2c] \setminus \mathcal{S}(G).$$ Then there exists $N=N(b) \in \N$  with the following property:

$$\mathrm{HF}^{b}_\bullet(G^-_{n}) \overset{\cong}{\longrightarrow} \mathrm{HF}^{b / \sigma}_\bullet(G^+_{n}), \quad \forall n \geq N(b),$$
induced by a concatenation of continuation morphisms. 
\end{lemma}

\begin{proof}
Let $[a,b] \subseteq (e,d-2c]$ and $N=N(a,b)$ be as in the above claim and the discussion preceding the claim. Then for every $v \in [a,b] \setminus \mathcal{S}(G)$ we can apply Lemma \ref{lemma non-crossing} since $[v-c_{n},v+c_{n}]\subseteq [a-c_{n},b+c_{n}]$. This has the following implication:  the whole strip $S \subseteq \R^2$ bounded by the graphs $(s,b(s))$ and $(s,a(s))$ does not belong to the spectrum of $\mathcal{S}(G_{n}^s)$, i.e.\ $$v(s) \notin \mathcal{S}(G_{n}^s), \quad \forall s \in [0,1], \, v \in [a,b].\footnote{This is nicely pictured in \cite[figure 3]{ms2011}.}$$ In particular, we can perform a ``zig-zag" adiabatic argument: for $s=0$ we have that all $v \in [a,b]$ satisfy
$$v=v(0) \notin \mathcal{S}(G_{n}^0)=\mathcal{S}(G_{n}^-).$$
That is, $[a,b]$ does not contain any critical value of $\mathcal{A}_{G^-_{n}}$, which implies
$$\mathrm{HF}^b_\bullet(G^-_{n})=\mathrm{HF}^a_\bullet(G^-_{n}).$$
Now we move from $(0,a(0))$ horizontally within our strip $S$ until we hit the upper graph of $b(s)$, i.e.\ $(s_1,a(0))=(s_1,b(s_1))$. Note that $a(0)=a=b(s_1)$. By definition of our strip $S$ we have $a(0) \notin \mathcal{S}(G_{n}^s)$ for $s \in [0,s_1]$, hence by standard Floer theory
$$\mathrm{HF}_\bullet^{a}(G^-_{n}) \cong \mathrm{HF}_\bullet^{a}(G^{s_1}_{n})=\mathrm{HF}_\bullet^{b(s_1)}(G^{s_1}_{n}).$$ As before, the choice of our strip $S$ implies that $[a(s_1),b(s_1)]$ does not intersect the spectrum of $G^{s_1}_n$ and thus we can ``drop" again:
$$\mathrm{HF}^{b(s_1)}(G^{s_1}_n)=\mathrm{HF}^{a(s_1)}(G^{s_1}_n).$$ Repeating the previous ``horizontal move" we see that the whole algorithm can be performed a finite amount of times until we reach $(1,w(1))$ for some $w \in[a,b]$ that is not necessarily equal to $b$. At that point we ``jump up" (analogously to the ``dropping") to $(1,b(1))=(1,b/\sigma)$ without affecting the Floer homology. This finishes the proof.
\end{proof}

\begin{remark}
The size of the interval $[a,b]$ used in the proof above does not matter --- we only needed a strictly smaller $a$ that allowed us to bounce in a ``zig-zag" fashion towards $b(1)=b/ \sigma$.
\end{remark}

Combining this with the isomorphisms described in Section \ref{section Spectral Invariants and Minimax Values} we obtain:

\begin{corollary}\label{corollary pinching diagram}

Let $b \in (e,d-2c] \setminus \mathcal{S}(G)$ and $N=N(b)$ as in Lemma \ref{lemma adiabatic argument}. Denote by $L_{n}^+$ the Lagrangian dual to $G^+_{n}=\sigma G + W_{n}$. Then for all $n \geq N$ we have the following commutative diagram:

\begin{center}
	\begin{tikzcd}
	& H_\bullet\left(\lbrace \mathcal{E}_{L_{n}^+} \leq b/\sigma \rbrace\right) \arrow{d}{\mathrm{AM}} \\
	& \mathrm{HM}_\bullet^{b/ \sigma}\left(\mathcal{E}_{L_{n}^+}\right) \arrow{d}{\mathrm{AS}} & \\
	\mathrm{HF}^{b}_\bullet\left(G_{n}^-\right) \arrow{r}{\cong}\arrow{dr} & \mathrm{HF}_\bullet^{b / \sigma}(G_{n}^+) \arrow{r} & \mathrm{HF}_\bullet^b\left(G_{n}^+\right) \\
	& \mathrm{HF}_\bullet^{b}\left(K_{n}\right) \arrow{ur}
	\end{tikzcd}
\end{center}

\end{corollary}

Denote by $L$ the dual Lagrangian to $G$, i.e.\ $$L(q,v)=\frac{1}{2} \Vert v \Vert^2,$$ and also observe that $L_{n}^+$ defined as in Corollary \ref{corollary pinching diagram} takes the form $$L_{n}^+(t,q,v)=\frac{1}{2\sigma} \Vert v \Vert^2-W_{n}(t,q).$$
We define $L^+ := \lim_{n \to \infty} L^+_n$ and observe
$$L^+=\frac{1}{\sigma} L.$$

\subsection{The Pinching Inequality}

 The next theorem can be viewed as the main building block of our main result.

\begin{theorem}\label{theorem spectral invariants relations}
Let $\alpha \in H_\bullet(\mathcal{L}Q) \setminus \{0\}$. Then

\begin{align*}
& c_\alpha(G^-)=\sigma \cdot c_\alpha(\sigma G),  \\
& c_\alpha(\sigma G) \leq c_\alpha(K) \leq c_\alpha(G^-).
\end{align*}

\end{theorem}

\begin{proof}
Let $b>0$ big enough such that $\alpha$ is in the image of $H_\bullet\left( \lbrace \mathcal{E}_L \leq b-c\sigma \rbrace \right) \to H_\bullet(\mathcal{L}Q).$\footnote{Note that by choice of $L$ we have $\mathcal{E}_L=\mathcal{E}$.} In particular, if $q \in \lbrace \mathcal{E}_L \leq b-c\sigma \rbrace$, then
$$\mathcal{E}_{L_{n}^+}(t,q) \leq \frac{1}{\sigma}\mathcal{E}_L(q)+c \leq \frac{b}{\sigma}, \quad \forall n \in \N.$$
Thus $\alpha$ can be viewed as an element in $H_\bullet \left( \lbrace \mathcal{E}_{L_{n}^+} \leq b / \sigma \rbrace \right)$. Accordingly, we take $d$ so big and $c$ so small that $$b+\sigma c< d-2c.$$ Up to slightly perturbing $b$ we may also assume $b \notin \mathcal{S}(G)$ and $b > e$ (if not, increase $b$ and $d$).
By Corollary \ref{corollary pinching diagram} we can send $\alpha$ through the commutative diagram and (abusively) still denote by $\alpha$ the Floer/Morse class in the corresponding group for each $n \geq N(b)$. Proposition \ref{proposition spectral invariant for degenerate Hamiltonians } and the choice of $b$ tells us that
\begin{align*}
\lim_{n\to \infty}c_\alpha(G^-_{n})&=c_\alpha(G^-) \leq b, \\
 \lim_{n \to \infty}c_\alpha(K_{n})&=c_\alpha(K) \leq b, \\
 \lim_{n \to \infty}c_\alpha(G_{n}^+)&=c_\alpha(\sigma G) \leq \frac{b}{\sigma}.
\end{align*}
These limits allude to the proof strategy: let the action filtration $b$ run towards $c_\alpha(G^{-})$ and use the isomorphism
$\mathrm{HF}^b_{\bullet}(G^-_n) \cong \mathrm{HF}^{b/\sigma}_\bullet(G^+_n)$ in order to conclude $c_\alpha(G^+) \leq \sigma c_\alpha(G^-)$ and then use symmetry of the argument for the other inequality. This however, also requires to let $n$ go to infinity, while $n \geq N(b)$ has to be satisfied. Define
$$B_{m}=c_\alpha(G^-)+c_{m}, \quad C_{m}=c_\alpha(K)+c_{m}, \quad D_{m}=\sigma \cdot \left( c_\alpha(\sigma G) +c_{m} \right).$$
We claim that $$B_m \in (e,d-2c], \quad \forall m \in \N$$ holds so that we can apply Corollary \ref{corollary pinching diagram} with action filtration $B_m$. By definition of $B$ and $c_m<c$ we get $$B_m \leq b+c<d-2c \quad \forall m  \in \N.$$ By spectrality, i.e.\ Proposition \ref{proposition spectral invariant for degenerate Hamiltonians }: $$c_\alpha(G^-)=\mathcal{A}_{G^-}(x) \in \mathcal{S}(G^-).$$ The above bound and the first bullet point in Proposition \ref{proposition control of gamma} for $G^-$ imply 
$$ G(x) \leq d.$$ In particular, $\tau_{c+d}$ vanishes and hence $ G(x)=G^-(x)$ with $x \in \mathcal{P}(G)$, but for $c$ sufficiently small\footnote{Choose $c$ so small that $e=4(\sigma+4)c$ lies below the non-zero spectrum of $G$.} the spectrum of $G$ is strictly greater than $e$, thus
$$e < G(x)=G^-(x)=c_\alpha(G^-) \leq B_m,$$ which proves the lower bound and thus that the sequence $B_m$ is contained in $(e,d-2c]$. Up to perturbing $c_m$ we may thus assume
$$B_m \in (e,d-2c] \setminus \mathcal{S}(G)$$ and can therefore apply Lemma \ref{lemma adiabatic argument} to every single $B_m$:
 
$$\mathrm{HF}^{B_m}_\bullet(G^-_{n_m}) \overset{\cong}{\longrightarrow} \mathrm{HF}_\bullet^{B_{m} / \sigma}(G^+_{n_m}), \quad \forall n_m \geq N(B_m).$$
Pick $(n_m)_{m \in \N}$ a strictly increasing sequence with $n_m \geq m$. Using Proposition \ref{proposition spectral invariant is C0 lipschitz continuous} and Proposition \ref{proposition spectral invariant for degenerate Hamiltonians } we get
\begin{align*}
\vert c_\alpha(G_{n_m}^-)-c_\alpha(G^-) \vert&=\lim_{k \to \infty} \vert c_\alpha(G_{n_m}^-)-c_\alpha(G_{k}^-) \vert \\
&\leq \lim_{k \to \infty} \Vert W_{n_m}-W_{k} \Vert_{C^0} \\
&\leq \Vert W_{n_m} \Vert_{C^0} \\
&<c_{n_m},
\end{align*}
thus
$$c_\alpha(G^-_{n_m}) \leq c_\alpha(G^-)+c_{n_m}\leq c_\alpha(G^-)+c_{m}=B_{m}.$$
This means that the LHS of the morphism above ``sees" $\alpha$, and so does the RHS. More precisely, we have the commutative diagram 
\begin{center}
	\begin{tikzcd}
	\mathrm{HF}_\bullet^{B_{m}}(G^-_{n_m}) \arrow{r}{\cong}\arrow{d}{i} & \mathrm{HF}_\bullet^{B_{m} / \sigma}(G^+_{n_m}) \arrow{d}{i} \\
	\mathrm{HF}_\bullet(G^-_{n_m}) \arrow{r}{\cong} & \mathrm{HF}_\bullet(\sigma G^+_{n_m}),
	\end{tikzcd}
\end{center}
where $\alpha$ belongs to the image of the left vertical map. By commutativity of the diagram and the definition of spectral invariants we obtain
$$c_\alpha(G^+_{n_m}) \leq \frac{B_m}{\sigma}=\frac{c_\alpha(G^-)+c_m}{\sigma}.$$ 
Taking the limit as $m$ goes to infinity then proves:
$$\sigma \cdot c_\alpha(\sigma G) \leq  c_\alpha(G^-).$$ 

For the other inequality we consider $D_{m}$ defined as above and claim again, as before, that 
$$D_m \in (e,d-2c], \quad \forall m \in \N.$$ Proposition \ref{proposition floer and morse minimax agree} and the choice of $b$ imply $c_\alpha(G^+_{n})=c_\alpha(L^+_{n}) \leq b/\sigma$ and thus we get $$c_\alpha(\sigma G)=c_\alpha(L^+) \leq b/ \sigma$$ by Proposition \ref{proposition spectral invariant for degenerate Hamiltonians }. Therefore $$D_m \leq b+\sigma c_{m} \leq b+\sigma c< d-2c, \quad \forall m \in \N.$$ For the lower bound observe $\sigma c_\alpha(\sigma G) \leq D_m$, but $$c_\alpha(\sigma G) \in \mathcal{S}(\sigma G)=\frac{1}{\sigma}\mathcal{S}(G),$$ by Lemma \ref{lemma summary Heistercamp}. As before, we may assume that $e$ bounds $\mathcal{S}(G) \cap (0,+\infty)$ from below. In particular we end up with $$e < \sigma c_{\alpha}(\sigma G) \leq D_m,$$ which proves the lower bound and hence the claim.
We are now in a position to apply Lemma \ref{lemma adiabatic argument} again (with reversed arrow and after perturbing $c_m$):
$$\mathrm{HF}^{\frac{D_{m}}{\sigma}}_\bullet(G^+_{n_m}) \overset{\cong}{\longrightarrow} \mathrm{HF}_\bullet^{D_{m}}(G^-_{n_m}), \quad \forall n_m \geq N(D_m).$$
As before we obtain
$$c_\alpha(G^+_{n_m}) \leq c_\alpha(\sigma G)+c_{n_m}\leq c_\alpha(\sigma G)+c_{m}=D_m.$$
Thus the LHS above sees $\alpha$, and thus the analogous commutative diagram grants
$$c_\alpha(G^-_{n_m}) \leq D_{m}=\sigma\left(c_\alpha(\sigma G)+c_{m}\right).$$
Taking the limit and also applying the previous inequality gives:
$$c_\alpha(G^-) \leq  \sigma c_\alpha( \sigma G) \leq c_\alpha(G^-).$$ This proves the desired first equality.

The inequalities are less involved: we use yet another commutative diagram:

\begin{center}
	\begin{tikzcd}
	\mathrm{HF}_\bullet^a(G^-_{n}) \arrow{r}\arrow{d}{i} & \mathrm{HF}_\bullet^a(K_{n}) \arrow{d}{i} \\
	\mathrm{HF}_\bullet(G^-_{n}) \arrow{r}{\cong} & \mathrm{HF}_\bullet(K_{n}).
	\end{tikzcd}
\end{center}
The same reasoning as in the previous two cases above gives us
$$c_\alpha(K_{n}) \leq c_\alpha(G^-_{n}), \quad \forall n \in \N.$$
Taking the limit proves
$$c_\alpha(K) \leq c_\alpha(G^-).$$ The other inequality
$$c_\alpha(\sigma G) \leq c_\alpha(K)$$ is proved the same way and thus the proof is complete.
\end{proof}

\begin{remark}\label{remark subtlety in proof}
The crux of the proof of Theorem \ref{theorem spectral invariants relations} comes from the fact that the filtration isomorphism from Lemma \ref{lemma adiabatic argument} does not hold for all $n$, i.e.\ changing level $b$ might increase $N(b)$. This forces us to use the somewhat convoluted sequence argument instead of just reading off the spectral invariant equality as done in Proposition \ref{proposition floer and morse minimax agree}.
\end{remark}

Using Proposition \ref{proposition floer and morse minimax agree} and Proposition \ref{proposition LS spectral invariant and limit for degenerate Lagrangians agree} we obtain
$$\sigma c_\alpha(\sigma G)=\sigma c_\alpha(L^+)=\sigma c_\alpha(L / \sigma).$$
It readily follows from the definition of minimax values that $$\sigma c_\alpha\left(L / \sigma\right)=c_\alpha(L), \quad \forall \alpha \in H_\bullet(\mathcal{L}Q) \setminus \{0\}.$$
Combining these with Theorem \ref{theorem spectral invariants relations} grants:

\begin{corollary}\label{corollary - spectral invariants do not depend on d, middle ones have bounded range}
Let $\alpha \in H_\bullet\left( \mathcal{L}Q \right) \setminus \{0\}$. Then the spectral invariants $c_\alpha(G^-)$ do not depend on $d$ and satisfy:
$$c_\alpha(G^-)=\sigma \cdot c_\alpha(\sigma G)=c_\alpha(L).$$
In particular
$$c_\alpha(K) \leq c_\alpha(L).$$
\end{corollary}

\section{Index Relations}\label{section index relations}

\subsection{The Robbin--Salamon Index}
One of the key steps in the proof of Theorem \ref{theorem main section existence} is the use of an index iteration formula due to M. de Gosson, S. de Gosson and Piccione (GGP) \cite{ggp2008}. In order to correctly apply the latter, one needs some a priori control on the Conley--Zehnder/Robbin--Salamon index of the carriers of some carefully chosen spectral invariants. In this section we mainly recall some main properties of the Robbin--Salamon index \cite{rs1993} and describe how the index of carriers of $c_{\alpha}(K)$ behave for any non-zero cohomology class $\alpha \in H_i(\mathcal{L}Q)$ with respect to the degree $i \in \N$.

The Robbin--Salamon index picks two Lagrangian paths on a fixed $2n$-dimensional symplectic vector space, and associates to it a value in $\frac{1}{2} \Z$. For a path $\Gamma \in \mathcal{S}(2n)$, i.e.\
$$\Gamma \, \colon [0,1] \longrightarrow \mathrm{Sp}(2n), \quad \Gamma(0)=\mathrm{id},$$
one defines
$$\mu_{\mathrm{RS}}(\Gamma)=\mu_{\mathrm{RS}}(\mathrm{gr}(\Gamma),\Delta),$$
where $\mathrm{gr}(\Gamma)$ is the graph of $\Gamma$ in the symplectic vector space $\left(\R^{2n} \times \R^{2n},-\omega\oplus \omega\right)$ and $\Delta$ is the diagonal in the latter. With this definition, one obtains

$$\mu_{\mathrm{CZ}}(\Gamma)=\mu_{\mathrm{RS}}(\Gamma), \quad \forall \Gamma \in \mathcal{S}^*(2n)\, :=\left\lbrace \Gamma \in \mathcal{S}(2n), \; \big | \, \det(\Gamma(1)-\mathrm{id}) \neq 0 \right \rbrace.$$
Many authors, GGP included, define the Conley--Zehnder index like this in the first place. There are several advantages of the Robbin--Salamon index, one being that there is no ambiguity whenever $\Gamma$ is a degenerate symplectic path, i.e.\ $\det(\Gamma(1)-\mathrm{id})=0$.

The following proposition is a standard property of the Robbin--Salamon index, which we will use:

\begin{proposition}\label{proposition RS index property}
Let $\Gamma \in \mathcal{S}(2n)$ and $\Theta \, \colon S^1 \to \mathrm{Sp}(2n)$ a loop. Then
$$\mu_{\mathrm{RS}}(\Theta \cdot \Gamma)=\mu_{\mathrm{RS}}(\Gamma)+2\mu_{\mathrm{Maslov}}(\Theta).$$
In particular, whenever $\Theta$ is homotopic to the constant loop $\mathrm{id}$, it holds
$$\mu_{\mathrm{RS}}(\Theta \cdot \Gamma)=\mu_{\mathrm{RS}}(\Gamma).$$
\end{proposition}

Now for $T$-periodic Hamiltonian orbits $x$ of $F \, \colon T^*Q \to \R$ we define 

$$\mu_{\mathrm{RS}}(x)\, :=\mu_{\mathrm{RS}}(\Gamma_x), \quad \Gamma_x(t)=\Phi_x(t)^{-1} \circ d\varphi_H^t(x(0)) \circ \Phi_x(0),  \; t \in [0,T],$$
where
$$\Phi_x \, \colon \R / T \Z \times \R^{2n} \longrightarrow x^*T(T^*Q)$$
is any vertically preserving symplectic trivialization of $x$. The argument in \cite[ Lemma 1.3]{as2006} carries over to the Robbin--Salamon index and shows that the choice of vertically preserving symplectic trivialization does not affect the index, hence $\mu_{\text{RS}}(x)$ as above is well defined.

To apply the results of GGP we will view every $1$-periodic Hamiltonian orbit $x = (q, p) \in \mathcal{P}(F)$ as a periodic Reeb/Hamiltonian
orbit $x_{1/\sqrt{a}}$ on $\Sigma$ (here $a = F(x)$, see Section \ref{section Reeb Orbits on Starshaped Domains}). We show that this shift does not affect the
Robbin--Salamon index:

Let $x \in \mathcal{P}(F)$. For every $s>0$, the orbit $x_s(t)=(q(st), s \cdot p(st))$ is a $1/s$-periodic Hamiltonian orbit of $F$; cf.\ Proposition \ref{proposition 1-per somewhere is same as per but on sigma}. Define the map

$$c_s \, \colon T^*Q \longrightarrow T^*Q, \quad (q,p) \mapsto (q, s\cdot p).$$
For $x_s$ we define the vertical trivialisation
$$\Phi_{x_s} \, \colon \R / s^{-1} \Z \times \R^{2n} \longrightarrow (x_s)^*T(T^*Q), \quad \Phi_{x_s}(t,\cdot)\, :=dc_s(x(st)) \circ \Phi_x(st,\cdot).$$
In particular, $\mu_{\text{RS}}(x_s)$ can be computed as the Robbin--Salamon index of the path $\Gamma_{x_s}$ defined as
$$\Gamma_{x_s} \, \colon [0,1/s] \longrightarrow \mathrm{Sp}(2n), \quad \Gamma_{x_s}(t)\, :=\Phi_{x_s}(t)^{-1} \circ d \varphi_F^t(x_s(0)) \circ \Phi_{x_s}(0).$$

\begin{proposition}\label{proposition p-shift does not change index v2}
For all $s>0$, $x \in \mathcal{P}(F)$ and $\Gamma_x$, $\Gamma_{x_s}$ as above, it holds:

$$\Gamma_{x_s}(t)=\Gamma_x(st), \quad \forall t \in \R / s^{-1} \Z.$$
In particular,
$$\mu_{\mathrm{RS}}(x_s)=\mu_{\mathrm{RS}}(x), \quad \forall s >0.$$
\end{proposition}

\begin{proof}
The first identity follows from the definitions and the relation
$$c_{1/s} \circ \varphi_F^t \circ c_s=\varphi_F^{st}.$$
Indeed, $c_{1/s}$ is the inverse of $c_s$ and hence with the chain rule we are able to compute
\begin{align*}
\Gamma_{x_s}(t)&=\Phi_{x_s}(t)^{-1} \circ d\varphi_F^t(x_s(0)) \circ \Phi_{x_s}(0) \\
&=\left(\Phi_{x}(st)^{-1}\circ dc_s(x(st)) ^{-1}\right) \circ d\varphi_F^t(x(0)) \circ \left( dc_s(x(0)) \circ \Phi_x(0) \right) \\
&=\Phi_x(st)^{-1} \circ d\left(c_{1/s} \circ \varphi_F^t \circ c_s\right)(x(0)) \circ \Phi_x(0) \\
&=\Phi_x(st)^{-1}\circ d\varphi_F^{st}(x(0)) \circ \Phi_x(0) \\
&=\Gamma_x(st).
\end{align*}

The last assertion follows from the first identity:
$$\mu_{\mathrm{RS}}(x)=\mu_{\mathrm{RS}}(\Gamma_x)=\mu_{\mathrm{RS}}(\Gamma_{x_s})=\mu_{\mathrm{RS}}(x_s).$$
\end{proof}

Proposition \ref{proposition p-shift does not change index v2} tells us that viewing $1$-periodic orbits $y$ of $F$ on $\Sigma$ does not affect the Robbin--Salamon index, in other words, the $p$-shift does not affect the index. 

Note that $z(t)\, :=x(t+s)$ is also a $T$-periodic Hamiltonian orbit of $F$ if $x$ is. Similar to the $p$-shift index invariance, we claim that the time-shift does not affect the Robbin--Salamon index. This is the content of the Proposition \ref{proposition time-shift does not change index} below).

\begin{proposition}\label{proposition time-shift does not change index}
For $x(t)$, $z(t)=x(t+s)$, and $s \in [0,T]$ as above, we have
$$\mu_{\mathrm{RS}}(\Gamma_z^1)=\mu_{\mathrm{RS}}(\Gamma^2_z).$$
In particular
$$\mu_{\mathrm{RS}}(x)=\mu_{\mathrm{RS}}(z).$$
\end{proposition}

\begin{proof}
The idea is to use Proposition \ref{proposition RS index property} with the loop
$$\Theta(t)\colon=\Gamma_z^2(t)\cdot \Gamma_z^1(t)^{-1}, \quad t \in \R / T\Z.$$
Observe that $\Theta$ is a loop based at the identity. Moreover
\begin{align*}
\Theta(t)&=\left[\Psi_z^2(t)^{-1} \circ d\varphi_F^t(z(0)) \circ \Psi_z^2(0) \right]\circ \left(\Psi_z^1(0)^{-1} \circ d\varphi_F^t(z(0))^{-1} \circ \Psi_z^1(t)^{-1}\right) \\
&=\left[\Phi_x(t)^{-1} \circ d\varphi_F^s(x(t))^{-1}\circ d\varphi_F^t(x(s)) \circ d\varphi_F^s(x(0)) \circ \Phi_x(0) \right] \circ \\
& \quad \quad \circ \Phi_x(s)^{-1} \circ d\varphi_F^t(x(s))^{-1} \circ \Phi_x(t+s) \\
&=\left[\Phi_x(t)^{-1} \circ d\varphi_F^{-s}(x(t+s)) \circ d\varphi_F^{t+s}(x(0)) \circ \Phi_x(0) \right] \circ \\
& \quad \quad \circ \Phi_x(s)^{-1} \circ d\varphi_F^t(x(s))^{-1} \circ \Phi_x(t+s) \\
&=\left[\Phi_x(t)^{-1} \circ d\varphi_F^t(x(0)) \circ \Phi_x(0) \right] \circ \Phi_x(s)^{-1} \circ d\varphi_F^t(x(s))^{-1} \circ \Phi_x(t+s). \\
\end{align*}
Setting $s=0$ in the expression above gives
$$\Phi_x(t)^{-1} \circ \Phi_x(t)=\mathrm{id},$$
which then readily proves that $\Theta$ is homotopic to the constant loop $\mathrm{id}$. In particular, using Proposition \ref{proposition RS index property}, we get:
$$\mu_{\mathrm{RS}}(\Gamma_z^2)=\mu_{\mathrm{RS}}(\Theta \cdot \Gamma_z^1)=\mu_{\mathrm{RS}}(\Gamma^1_z).$$
At the same time, the above computation (see square brackets!) also proves $\Gamma_z^2(t)=\Gamma_x(t)$. In particular
$$\mu_{\mathrm{RS}}(z)=\mu_{\mathrm{RS}}(\Gamma_z^1)=\mu_{\mathrm{RS}}(\Gamma_z^2)=\mu_{\mathrm{RS}}(\Gamma_x)=\mu_{\mathrm{RS}}(x).$$
\end{proof}

Instead of using the Robbin--Salamon index to extend $\mu_{\mathrm{CZ}}$ to the whole $\mathcal{S}(2n)$, we could have used upper/lower semicontinuous extensions $\mu_{\mathrm{CZ}}^{\pm}$. For paths $\Gamma \in \mathcal{S}^*(2n)$ all three notions agree. On the whole $\mathcal{S}(2n)$ the following relations are well known
\begin{align*}
\mu^+_{\mathrm{CZ}}(x)+\mu^-_{\mathrm{CZ}}(x)&=2 \cdot \mu_{\mathrm{RS}}(x) \\
\mu^+_{\mathrm{CZ}}(x)&=\mu^-_{\mathrm{CZ}}(x)+\nu(x),
\end{align*}
where $\nu(x)$ is the nullity i.e.\ the geometric multiplicity of the eigenvalue $1$ of the matrix $\Gamma_x(1)$:
$$\nu(x)=\dim \ker(\Gamma_x(1)-\mathrm{id})=\dim \ker (d  \varphi_F^1(x(0)) - \mathrm{id}),$$
 see \cite[Section 3.1]{am2017} for more details. Combining these two gives

$$\mu_{\mathrm{RS}}(x)=\mu^+_{\mathrm{CZ}}(x)-\frac{\nu(x)}{2}= \mu^-_{\mathrm{CZ}}(x)+\frac{\nu(x)}{2}.$$

With this we can bound the Robbin--Salamon index of the relevant carrier from above and below:

\begin{proposition}\label{proposition index bound from below and above}
Let $\alpha \in H_i(\mathcal{L}Q) \setminus\{0\}$ and $y$ a carrier of $c_\alpha(K)$. Then
$$i-\frac{\nu(y)}{2} \leq \mu_{\mathrm{RS}}(y) \leq i+\frac{\nu(y)}{2}.$$
\end{proposition}

\begin{proof}
Indeed, we have seen in the proof of Proposition \ref{proposition spectral invariant for degenerate Hamiltonians }  that $y$ is the limit of genuine non-degenerate orbits of Conley--Zehnder (and thus Robbin--Salamon) index $i$. Hence, by definition of $\mu^{\pm}_{\mathrm{CZ}}$:
$$\mu^+_{\mathrm{CZ}}(y) \geq i, \quad \mu^-_{\mathrm{CZ}}(y) \leq i.$$ This together with the identities above the proposition grants the desired bound.
\end{proof}

For $T$-periodic Reeb orbits $y$ of $\Sigma$ one defines the nullity of $y$ as follows:
$$\nu(y)=\dim\ker(d\varphi_R^T(y(0))-\mathrm{id}).$$

\begin{notation}
We adopt the following notation: if $x \in \mathcal{P}(F)$, then $\bar{x}$ denotes the corresponding $p$-shifted periodic Hamiltonian orbit of $F$ on $\Sigma$. The nullity $\nu(\bar{x})$ is understood to be the Reeb nullity as defined above --- this is not really abusive due to the next result.
\end{notation}

\begin{proposition}\label{proposition index inequality for reeb orbits}
Let $x \in \mathcal{P}(F)$. Then
$$\mu_{\mathrm{RS}}(\bar{x})=\mu_{\mathrm{RS}}(x) \text{ and } \nu(\bar{x})=\nu(x).$$
In particular, if $x$ is a carrier of $c_{\alpha}(K)$ as in Proposition \ref{proposition index bound from below and above}, then
$$i-\frac{\nu(\bar{x})}{2} \leq \mu_{\mathrm{RS}}(\bar{x}) \leq i+\frac{\nu(\bar{x})}{2}.$$
Moreover, the Reeb flow action on $\bar{x}$ does not affect the Robbin--Salamon index.
\end{proposition}

\begin{proof}
The first equation holds true because of Proposition \ref{proposition p-shift does not change index v2}, whereas the nullity equality follows from $F$ being $2$-homogeneous and a standard computation of $d\varphi_F^t$, see \cite[Lemma 3.3]{bo2009} for more details. The inequality now follows from Proposition \ref{proposition index bound from below and above}. The Reeb flow action on $\bar{y}$ is just a time-shift, which again does not affect the index; cf.\ Proposition \ref{proposition time-shift does not change index}.
\end{proof}

\begin{remark}\label{remark ggp and rs index agree}
The Conley--Zehnder index defined in \cite{ggp2008} agrees with the Robbin--Salamon index above; cf.\ \cite[Proposition 9]{gosson2009} and \cite[formula (39)]{gosson2009}.
\end{remark}

\section{Main Theorem}\label{section main theorem}

\subsection{Preliminaries for the Main Theorem and Examples}

Let $Q$ be a $n$-dimensional closed and connected manifold whose first Betti number does not vanish. Then there exists a homotopy class $\eta \in \pi_1(Q)$ whose image in $H_1(Q;\Z)$ is of infinite order. In particular, $\eta$ itself is of infinite order and the conjugacy classes of each iterate $\eta^m$ are distinct. Denote by $\mathcal{L}_mQ$ the connected component corresponding to the conjugacy class of $\eta^m$.

\begin{remark}
The choice of $\eta$ above implies that $\mathcal{L}_{k}Q \neq  \mathcal{L}_{l}Q$ for integers $k \neq l$. It is not clear whether an arbitrary element $\beta \in \pi_1(Q)$ of infinite order has the same implication. A similar issue regarding conjugacy classes is addressed in \cite[page 4]{fs2004}. The same Betti number assumption is made by Allais in \cite{allais2020}. Note that Allais proves, among other things, a version of Theorem \ref{theorem main section existence} in the case of geodesic chords.
\end{remark}

We fix once and for all a homotopy class $\eta \in \pi_1(Q)$ as above.

\begin{definition}\label{definition S1 non trivial action}
A continuous $S^1$-action $\phi \, \colon S^1 \times Q \to Q$ is called non-trivial with respect to $\beta \in \pi_1(Q,q_0)$, if  $[\phi(\cdot,q_0)]=\beta$.
\end{definition}
Note that any two orbits of $\phi$ have the same free homotopy class since $Q$ is assumed to be connected.

\begin{convention}
Whenever $Q$ admits a homotopy class $\eta$ as described above, we (abusively) call an $S^1$-action $\phi \, \colon S^1 \times Q \to Q$ \emph{non-trivial} if there exists a non-zero integer $k$ such that $\phi$ is non-trivial with respect to $\eta^k$ in the sense of Definition \ref{definition S1 non trivial action}. Note that up to switching $\eta$ and $\eta^{-1}$, we may assume $k> 0$.
\end{convention}

Let $\Sigma$ be a starshaped hypersurface as in the previous sections. We define what it means for $\Sigma$ to be non-degenerate:

\begin{definition}\label{definition non-degenerate hypersurface}
Let $y$ be a $T$-periodic Reeb orbit of $\Sigma$ with $T>0$. Then $y$ is said to be non-degenerate if
$$\det\left(d(\varphi_R^T)\big |_{\xi_{y(0)}}-\mathrm{id}_{\xi_{y(0)}}\right) \neq 0,$$
where the contact structure $\xi$ is defined as in Section \ref{section Reeb Orbits on Starshaped Domains}. A starshaped hypersurface $\Sigma$ is said to be non-degenerate if every non-constant Reeb orbit of $\Sigma$ is non-degenerate.

\end{definition}

\begin{remark}\label{remark non-degeneracy of Sigma}
Non-degeneracy of $\Sigma$ ensures that the nullity of every Reeb orbit $y$ is precisely $1$. In parcticular
$$\nu(x)=\nu(\bar{x})=1, \quad \forall x \in \underline{\mathcal{P}}(F),$$
cf.\ Proposition \ref{proposition index inequality for reeb orbits}. Controlling the nullity will be crucial in the proof of Proposition \ref{proposition iteration distinct} and thus Theorem \ref{theorem main section existence}.

\end{remark}

 \begin{theorem}\label{theorem main section existence}
 Let $Q$ be a $n$-dimensional closed connected manifold with $n \geq 2$ whose first Betti number is non-trivial. Assume that $Q$ admits a non-trivial $S^1$-action.
 Then for any starshaped hypersurface $\Sigma \subseteq T^*Q$, whose closed Reeb orbits have nullity less or equal to $n-1$, it holds:
 $$\mathcal{N}_\Sigma(T) \geq \frac{1}{\log(2n)} \cdot \log\left(aT-1\right) , \quad \forall T >0$$
 for some $a>0$. 

 \end{theorem}
 
\begin{remark}\label{remark S1 action suffices}
The non-trivial action assumption can be weakened to the existence of a section $$s \, \colon Q \longrightarrow \mathcal{L}_kQ,$$ for some $k \neq 0$. Note that the formula $\phi(t,q)\, :=s_k(q)(t)$ does not necessarily define an $S^1$-action. 
\end{remark} 
 
In view of Remark \ref{remark non-degeneracy of Sigma} we obtain the following corollary:

\begin{corollary}\label{corollary Sigma non-degenerate}
Let $Q$ be as in Theorem \ref{theorem main section existence}. Then for any non-degenerate starshaped hypersurface $\Sigma \subseteq T^*Q$ we have
$$\liminf_{T \to \infty} \frac{\mathcal{N}_\Sigma(T)}{\log(T)}>0.$$
\end{corollary}

Here is a list of examples that satisfy the hypotheses of Theorem \ref{theorem main section existence}:

\begin{enumerate}[i)]

\item  any product $S^1 \times M$ with $M$ a closed manifold of dimension at least $1$,
\item any principal $S^1$-bundle $Q$ over a closed connected base manifold $B$ with $\pi_2(B)=0$ and $\dim B \geq 1$,
\item any $H$-space $Q$ with $\dim Q \geq 2$ and infinite fundamental group.

\end{enumerate}

The last two examples require some explanation. Ad ii): From the homotopy long exact sequence we deduce that the fiber action $\varphi \, \colon S^1 \to Q$ induces an injection on the fundamental groups and therefore defines a non-trivial $S^1$-action --- see also \cite[Proposition 3.1]{macarini2004}.  Since principal $S^1$-bundles over $B$ are parametrized by $H^2(B;\Z)$  \cite{kobayashi1956}, this gives a whole family of examples for any fixed $B$ meeting the above conditions.

Ad iii): The fundamental group $\pi_1(Q,e)$ is abelian since $(Q,\cdot)$ is an $H$-space. Thus $\pi_1(Q,e) \cong H_1(Q;\Z)$ by Hurewicz, but the latter is a finitely generated $\Z$-module since $Q$ is a closed manifold, therefore $\pi_1(Q,e)$ contains a $\Z$-copy generated by some homotopy class $\eta$. Let $\gamma$ be a representative of $\eta$ and denote by $\mathcal{L}_1Q$ its loop space component. Associated to $\gamma$ we have a section
$$s \, \colon Q \longrightarrow \mathcal{L}_1Q, \; s(q):= q \cdot \gamma(\cdot).\footnote{Note that each loop $s(q)$ is freely homotopic to $s(e)=\gamma$ since $Q$ is path connected.}$$ While this section is not necessarily induced from an $S^1$-action on $Q$, we will shortly see that its existence is still sufficient to run the proof of Theorem \ref{theorem Main Theorem Introduction} and Corollary \ref{corollary Sigma non-degenerate}, see Remark \ref{remark S1 action suffices}.
 
Let $Q$ be as Theorem \ref{theorem main section existence}. We proceed as in the introduction and define 
$$\alpha_m:=(\mathcal{I}_m)_*\alpha_1 \in H_n(\mathcal{L}_{mk}Q) \setminus\{0\}.$$ Recall that we deduced $\alpha_m \neq 0$ using the non-trivial section $s \, \colon Q \to \mathcal{L}_kQ$. 
Finally define
$$\tau_m\, :=\sqrt{c_{\alpha_m}(K)}, \; \tau_m^{\pm}\, :=\sqrt{c_{\alpha_m}(G^{\pm})}.\footnote{We are suppressing $d$ from the notation of $\tau_m$ once again.}$$

The following lemma will be key for the growth control and is based on the pinching machinery developed in Section \ref{section Pinching}.

\begin{lemma}\label{lemma iteration spectral bound}
We have 
$$\tau_{m}\leq m \cdot \tau_1^-=  m \cdot  \sqrt{c_{\alpha_1}(L)}, \quad \forall m \in \N.$$
\end{lemma}

\begin{proof}
First of all $$c_{\alpha_m}(K) \leq c_{\alpha_m}(G^-)=c_{\alpha_m}(L),$$
by Corollary \ref{corollary - spectral invariants do not depend on d, middle ones have bounded range}.
Let $\left(\sigma^i\right)_{i \in \N}$ be a sequence of representatives of the singular homology class $\alpha_1 \in H_n(\mathcal{L}_kQ)$ such that
$$c_{\alpha_1}(L)=\lim_{i \to \infty} \sup_{v \in \Delta^n}\mathcal{E}\left(\sigma^i(v)\right)=\mathcal{E}(q_1),\footnote{Here we are implicitly assuming that $\sigma^i$ is not a sum of several $n$-simplices --- otherwise we would need to take the supremum of $\mathcal{E}$ over the union of the images of all summands. This assumption is purely for notational convenience.}$$
with $q_1$ a critical point of $\mathcal{E}$ inside $\mathcal{L}_kQ$  --- such a sequence and critical point exist by classical LS theory. By definition we know that $$\mathcal{I}_m \circ \sigma^i \, \colon \Delta^n \to \mathcal{L}_{mk}Q$$
defines a representative of $\alpha_m$. In particular
$$c_{\alpha_m}(L) \leq \lim_{i \to \infty}\sup_{v \in \Delta^n} \mathcal{E}(\mathcal{I}_m \circ \sigma^i(v)).$$
However, for every $v \in \Delta^n$ we have
\begin{align*}
\mathcal{E}(\mathcal{I}_m \circ \sigma^i(v))=\mathcal{E}(\underbrace{\sigma^i(v) * \dots * \sigma^i(v)}_{m\text{-times}})=m^2 \cdot \mathcal{E}(\sigma^i(v)),
\end{align*}
therefore by choice of $\sigma^i$ we obtain
$$c_{\alpha_m}(L) \leq m^2 \cdot \lim_{i \to \infty} \sup_{v \in \Delta^n}\mathcal{E}(\sigma^i(v)) =m^2 \cdot c_{\alpha_1}(L).$$
Taking square roots and putting everything together we see that
$$\tau_{m}=\sqrt{c_{\alpha_m}(K)}\leq \sqrt{c_{\alpha_m}(L)} \leq m \sqrt{c_{\alpha_1}(L)}=m \cdot \tau_1^-,$$
which concludes the proof.
\end{proof}

Now denote by $$x_m \in \mathcal{L}_{mk}T^*Q$$ 
the carrier of $\tau_m$ and recall the notation $\bar{x}_m$ from Section \ref{section index relations} to denote the corresponding Reeb orbit on $\Sigma$. Since $\alpha_m$ lives in degree $n$, Proposition \ref{proposition index inequality for reeb orbits} implies
$$n-\frac{\nu(\bar{x}_m)}{2} \leq \mu_{\mathrm{RS}}(\bar{x}_m) \leq n+\frac{\nu(\bar{x}_m)}{2}.$$

\begin{notation}
If $\bar{x} \, \colon [0,T] \to T^*Q$ is $T$-periodic, we define  $\bar{x}^r=\underbrace{\bar{x} \diamond \cdots \diamond \bar{x}}_{r-\mathrm{times}} \, \colon [0,rT] \to T^*Q$ to be the $r$-th iterate of $\bar{x}$ as in \cite{ggp2008}.
\end{notation}

The homotopy classes of $\bar{x}_m$ are not enough to distinguish the carriers from each other --- it could very well happen that the images of $\bar{x}_m^r$ and $\bar{x}_{rm}$ are the same for infinitely many $r \in \N$. The next result shows that this scenario does not occur under an additional assumption on the nullity of the Reeb orbits.

\begin{proposition}\label{proposition iteration distinct}
Assume that all the closed Reeb orbits $y$ on $\Sigma$ satisfy $$\nu(y) \leq n-1.$$ Let $r \in \N$, and $\bar{x}_m, \, \bar{x}_{rm}$ corresponding to carriers $x_m, \, x_{rm}$ of $c_{\alpha_m}(K)$, $c_{\alpha_{rm}}(K)$, respectively. Then the following implication holds true:
$$r \geq 2n \implies \mathrm{im}\left(\bar{x}^r_m\right) \neq \mathrm{im}(\bar{x}_{rm}).$$
\end{proposition}

\begin{proof}
We show the contrapositive: assume that $\bar{x}_m^r, \, \bar{x}_{rm}$ have the same image. In particular, the moreover part of Proposition \ref{proposition index inequality for reeb orbits} implies
$$\mu_{\mathrm{RS}}\left(\bar{x}^r_m\right)=\mu_{\mathrm{RS}}(\bar{x}_{rm}).$$

If $\Phi_{\bar{x}_m}$ is a vertically preserving symplectic trivialization of $\bar{x}_m\, \colon [0,T] \to T^*Q$, then its periodic extension to $[0,rT]$ defines a vertically preserving symplectic trivialization of $\bar{x}_{m}^r$. This together with Remark \ref{remark ggp and rs index agree} tells us that we can apply \cite[Corollary 4.4]{ggp2008} to obtain
$$\big \vert \mu_{\mathrm{RS}}\left(\bar{x}^r_m\right) - r \cdot \mu_{\mathrm{RS}}(\bar{x}_m) \big \vert \leq \frac{n}{2}\cdot \left( r-1 \right).$$
Combining this with the previous index inequality and the nullity assumption gives

\begin{align*}
n+\frac{n-1}{2}  \geq \mu_{\mathrm{RS}}(\bar{x}_{rm})=\mu_{\mathrm{RS}}\left(\bar{x}^r_m\right) &\geq r \cdot \mu_{\mathrm{RS}}(\bar{x}_m)-\frac{n}{2}(r-1)\\
&\geq r \cdot \left(n-\frac{\nu(\bar{x}_m)}{2}-\frac{n}{2} \right) + \frac{n}{2} \\
&\geq r \cdot \left( \frac{n}{2}-\frac{n-1}{2}\right)+\frac{n}{2} \\
&=\frac{r}{2}+\frac{n}{2}.
\end{align*}
Therefore $n - \frac{1}{2} \geq \frac{r}{2}$ and thus $r < 2n$.
\end{proof}

\subsection{Proof of the Main Theorem}

We can finally present the proof of Theorem \ref{theorem main section existence}.
\begin{proof}[Proof of Theorem \ref{theorem main section existence}]

Let $T>0$ be arbitrary and $t=T/2$. Pick $m' \in \N$ the maximal integer such that 
$$m'\sqrt{c_{\alpha_1}(L)} \leq t.$$
Let $d>0$ be so big that $\tau_{m}=\tau_{m,d}$ is defined for all $m=1,\dots,m'$ and such that the corresponding carriers by $x_m$ are elements in $\mathcal{P}(F)$ --- this is possible thanks to Proposition \ref{proposition precomposing with f does not change orbits} and the fact that the $x_m$ are not constant because their homotopy class is non-trivial.
Note that by Lemma \ref{lemma iteration spectral bound} we have
$$\tau_m \leq t, \quad \forall m=1,\dots,m'.$$

Fix $r \in \N$ such that $r \geq 2n$. We may assume that $m'>>r$ as $r$ is fixed throughout the proof. Observe that for any $s \in \N$, $x^s_m$ has the same homotopy class as $x_{sm}$. The corresponding Reeb orbits $\bar{x}^{s}_m$ and $\bar{x}_{sm}$ however, are geometrically distinct if $s \geq r$, due to Proposition \ref{proposition iteration distinct} and choice of $r$. 

If $N \in \N$ denotes the unique integer such that
$$r^{(N-1)} \leq m' \leq r^N,$$
then we have $N$ pairwise geometrically distinct Reeb orbits, namely
$$\bar{x}_1, \, \bar{x}_r, \, \bar{x}_{r^2}, \dots, \, \bar{x}_{r^{N-1}},$$
again due to choice of $r$ and Proposition \ref{proposition iteration distinct}. By construction 
$$\sqrt{\mathcal{A}_F(x_m)}=\tau_m \leq t,$$
in particular the period of the Reeb orbits $\bar{x}_m$ is less than $T=2t$, see Section \ref{section Reeb Orbits on Starshaped Domains}, and therefore
$$\mathcal{N}_\Sigma(T) \geq N.$$ The only thing left to do is finding a lower bound for $N$. To this end observe that by maximality of $m'$ we have
$$(m'+1) \sqrt{c_{\alpha_1}(L)} \geq t,$$
thus
\begin{align*}
m' \geq \frac{t}{\sqrt{c_{\alpha_1}(L)}} -1=\frac{T}{2\sqrt{c_{\alpha_1}(L)}} -1
\end{align*}
Moreover,
$$N \geq \frac{\log(m')}{\log(r)}$$
and therefore
\begin{align*}
\mathcal{N}_\Sigma(T) \geq N  \geq \frac{1}{\log(r)} \cdot \log \left( \frac{T}{2\sqrt{c_{\alpha_1}(L)}}-1 \right).
\end{align*}
Since any $r \geq 2n$ is allowed we can set $r=2n$ and obtain the desired result.
\end{proof}

\printbibliography

\end{document}